\documentclass[10pt]{amsart}

\usepackage[foot]{amsaddr}
\usepackage{amsfonts,epsfig,fancyhdr,stmaryrd}
\usepackage{newlfont,amsfonts,amssymb,amsmath,amsthm,amsgen,amscd,datetime,dsfont,hyperref,tensor}
\usepackage[small,nohug,heads=littlevee]{diagrams}
\usepackage[normalem]{ulem}
\diagramstyle[labelstyle=\scriptstyle]
\usepackage{multicol,picinpar,enumerate,cleveref,mathabx}
\usepackage{euscript,color}

\DeclareMathOperator{\Span}{span}

\DeclareMathOperator{\Aut}{Aut}
\DeclareMathOperator{\diag}{diag}

\DeclareMathOperator{\Id}{Id}

\newcommand{\vf}{\varphi}

\def\p{\partial}

\newcommand{\ro}{\mathsf{P}}
\newcommand{\tM}{\widetilde{M}}
\newcommand{\tg}{\widetilde{g}}

\newcommand{\z}{\mathfrak{z}}

\renewcommand{\k}{\mathfrak{k}}

\newcommand{\x}{\veccy{x}}
\newcommand{\y}{\veccy{y}}

\newcommand{\so}{\mathfrak{so}}

\renewcommand{\O}{\mathbf{O}}

\newcommand{\RR}{R}
\newcommand{\W}{W}

\newcommand{\Ric}{Ric}
\newcommand{\scal}{\mathrm{scal}}
\newcommand{\I}{\mathrm{I}}
\newcommand{\tr}{\mathrm{tr}}
\newcommand{\1}{\mathbf{1}}
\newcommand{\e}{\mathrm{e}}
\renewcommand{\d}{\mathrm{d}}
\newcommand{\hg}{\widehat{g}}
\newcommand{\hM}{\widehat{M}}
\newcommand{\hR}{\widehat{\RR}}
\newcommand{\hric}{\widehat{\Ric}}
\newcommand{\hscal}{\widehat{\scal}}

\newcommand{\R}{\mathbb{R}} 
\newcommand{\Z}{\mathbb{Z}} 
\newcommand{\N}{\mathbb{N}} 
\newcommand{\C}{\mathbb{C}} 
\renewcommand{\S}{\mathbb{S}} 
\DeclareMathOperator{\Homoth}{\mathrm{Homoth}}
\DeclareMathOperator{\Conf}{Conf}

\DeclareMathOperator{\Iso}{\mathrm{Isom}}

\DeclareMathOperator{\id}{id}

\DeclareMathOperator{\E}{\mathbf{Euc}}
\DeclareMathOperator{\Hei}{\mathbf{Hei}}

\newcommand{\<}{\langle}
\renewcommand{\>}{\rangle}

\newcommand{\veccy}[1]{\boldsymbol{#1}} 
\newcommand{\cent}{C}

\theoremstyle{definition}
\newtheorem{definition}{Definition}[section]
\newtheorem{remark}[definition]{Remark}
\newtheorem{example}[definition]{Example}
\theoremstyle{plain}
\newtheorem{lemma}[definition]{Lemma}
\newtheorem*{lem*}{Lemma}
\newtheorem{proposition}[definition]{Proposition}
\newtheorem{corollary}[definition]{Corollary}
\newtheorem{theorem}[definition]{Theorem}

\numberwithin{equation}{section}

\begin{document}

\title
{Conformal transformations  of Cahen-Wallach spaces }

\date{ \today}
\thanks{
This work was supported by 
 the Australian Research
Council (Discovery Program DP190102360).
 }
\author{Thomas Leistner}\address[Thomas Leistner, corresponding author]{School of Mathematical Sciences, University of Adelaide, SA~5005, Australia, phone: +61 8 83136401}\email{thomas.leistner@adelaide.edu.au}
\author{Stuart Teisseire}\address[Stuart Teisseire]{Department of Mathematics, University of Auckland, Private Bag 92019, Auckland~1142, New Zealand}
\email{stuart.teisseire@gmail.com}


\begin{abstract}
We study conformal transformations of indecomposable Lorentzian symmetric spaces of non-constant sectional curvature, the so-called Cahen-Wallach spaces. When a Cahen-Wallach space is  conformally curved, its conformal transformations are homotheties. Using this we show that a conformal transformation of a conformally curved Cahen-Wallach space is essential if and only if it has a fixed point. Then we explore the possibility of properly discontinuous groups of conformal transformations acting  with a compact orbit space on a conformally curved Cahen-Wallach space. 
We show that any such group cannot centralise an essential homothety and that for Cahen-Wallach spaces of imaginary type must be contained within the isometries.

%
%
%
 \end{abstract}

\maketitle
\setcounter{tocdepth}{1}

\section{Introduction}
It is a remarkable feature of Lorentzian geometry that indecomposable Lorentzian symmetric spaces either have constant  sectional curvature or are universally covered by a Cahen-Wallach space, \cite{cahen-wallach70,Cahen98}. A {\em Cahen-Wallach space} is a 
Lorentzian manifold $(\R^{n+2},g_S)$ with $n\ge 1$ and  Lorentzian metric
\begin{equation}\label{cwmetric}g_S=2\d v\d t +S_{ij}x^ix^j\,(\d t)^2+ \delta_{ij}\d x^i \d x^j = 2\d v\d t +\x^\top S\x (\d t)^2+ \d \x^\top \d \x ,\end{equation}
where $(t,x^1, \ldots, x^n,v)=(t,\x,v)$ are coordinates on $\R^{n+2}$ and $S=(S_{ij})$ is a symmetric $(n\times n)$-matrix with non-zero determinant (using Einstein's summation convention). 
Whereas the constant sectional curvature spaces are Einstein and conformally flat, the Cahen-Wallach spaces in general are neither Einstein nor conformally flat. 
Motivated  by the Clifford-Klein program, one mays ask which compact manifolds arise as compact quotients  of an indecomposable Lorentzian symmetric space by a group of isometries. For the constant curvature spaces Calabi and Markus  \cite{CalabiMarkus62} have shown that a group that acts properly discontinuously on de Sitter space must be finite and hence cannot produce a compact quotient. Moreover, Kulkarni has shown \cite{Kulkarni81} that the universal cover of anti-de Sitter space admits a compact quotient if and only if its dimension is odd.
For Cahen-Wallach spaces this question  only recently has been answered by Kath and Olbrich in \cite{KathOlbrich15}, who gave a classification of groups of isometries of a Cahen-Wallach space  that act properly discontinuously and with a compact quotient. Together with the completeness results in \cite{carriere89,klingler96,leistner-schliebner13} this yields a classification of Lorentzian compact indecomposable  locally symmetric spaces.

One may  extend such questions to conformal structures, i.e.~to manifolds equipped with an equivalence class of conformally equivalent metrics. For example, 
one may consider the 
 conformal class of a Cahen-Wallach metric  and ask for groups of conformal transformation that yield interesting and perhaps compact quotients. The resulting manifold would be equipped with a conformal structure that is locally conformally equivalent to a Cahen-Wallach metric. 
The local conformal geometry of Cahen-Wallach spaces, and more generally pseudo-Riemannian symmetric spaces, in particular their Killing vector fields, have been studied in papers by Cahen and Kerbrat \cite{CahenKerbratPraet75,CahenKerbratPraet76,CahenKerbrat77,CahenKerbrat78,CahenKerbrat82}. Questions about the global conformal geometry and in particular the existence of compact conformal quotients to our knowledge have not been considered in the literature and in this paper we are going to address some of these questions.

For our first result, we follow the convention in \cite{KathOlbrich15} and say that a Cahen-Wallach space is of {\em imaginary type} if  $S$ is negative definite.

\begin{theorem}\label{intro-thm_no-homoth-quotients-of-imaginary-CW}
Let $(\R^{n+2},g_S)$ be a conformally curved Cahen-Wallach space of imaginary type and $\Gamma$ a subgroup of its conformal group. If $\Gamma$ acts properly discontinuously and with compact orbit space  $M=\R^{n+2}/\Gamma$, then
			 $\Gamma$ is a group of isometries. Consequently,  $M$ with the metric induced from $g_S$ is  locally isometric to $(\R^{n+2},g_S)$ and  its conformal group is equal to its isometry group.
		\end{theorem}
Consequently, the question about compact conformal quotients of Cahen-Wallach spaces of imaginary type is reduced to the case of compact isometric quotients and the results in \cite{KathOlbrich15}.
Our proof of Theorem~\ref{intro-thm_no-homoth-quotients-of-imaginary-CW} relies on an analysis of the group of conformal transformations of a Cahen-Wallach space.
First,  it is straightforward to show that a Cahen-Wallach metric $g_S$ is conformally flat (i.e.~with vanishing Weyl tensor) if and only if the matrix $S$ is a scalar matrix, see Proposition~\ref{Wflatprop}. Since Cahen-Wallach spaces are locally symmetric,  a  result in \cite[Proposition 2.1]{CahenKerbrat82} implies that the conformal group of a conformally curved Cahen-Wallach space is equal to its homothety group. This implies in particular, that for compact quotients by a group of isometries the conformal group is equal to the isometry group, see Corollary~\ref{locsymcompconf}. 
 The homothety group of a Cahen-Wallach space 
  is equal to $\Iso(\R^{n+2},g_S)\rtimes \R$, and 
 the isometry group $\Iso(\R^{n+2},g_S)$ is well-known \cite{cahen-wallach70, KathOlbrich15} to be isomorphic to
 \[
\Hei_n\rtimes \left(\E(1)\times \cent_{\O(n)}(S)\right),\]
where $\Hei_n$ is the $2n+1$-dimensional Heisenberg group, $\E(1)=\R\rtimes \Z_2$ is the Euclidean group in one dimension, and $\cent_{\O(n)}(S)$ is the centraliser of the matrix $S$ in $\O(n)$. For details about the group structure and the result, see Section~\ref{confcw-sec} and Corollary~\ref{curvedconf}.
Explicitly, each homothety of $(\R^{n+2},g_S)$ is given by 
		\begin{equation}\label{homoth}
\begin{pmatrix}t\\\veccy{x}\\v\end{pmatrix} \stackrel{\phi}{\longmapsto} \begin{pmatrix}\epsilon\, t+c \\ \e^sA\veccy{x}+\beta(t) \\ \epsilon \left(\e^{2s}v+b-\< \dot\beta(t),A\boldsymbol{x}+\frac12\beta(t)\>\right)\end{pmatrix},
			\end{equation}
	where $(c,\epsilon)\in\R\times\{\pm 1\}=\E(1)$, $A\in \cent_{\O(n)}(S)$,  $b\in \R$, $\beta:\R\to\R^n$ is a solution to $\beta''=S\beta$ and $\<.,.\>$ is the standard Euclidean inner product on $\R^n$. The boundedness of the functions $\beta$ for  Cahen-Wallach spaces of imaginary type will yield the result in Theorem~\ref{intro-thm_no-homoth-quotients-of-imaginary-CW}. For Cahen-Wallach spaces that are not of imaginary type (i.e.~when $S$ has at least one positive eigenvalue), the method of our proof does not immediately apply, and we are not able to answer the question whether there are properly discontinuous 		groups of homotheties acting with a compact orbit space. In Section~\ref{ex-sec} we illustrate the difficulties that arise when trying to construct an example of such group.

A motivation for studying conformal transformations of Cahen-Wallach spaces  also comes from rigidity questions in conformal geometry, namely the question for which conformal manifolds the group of conformal transformation is {\em essential}, i.e.~not contained in the isometry group of a metric in the conformal class.
Of course, by definition the group of conformal transformations of a semi-Riemannian manifold  is larger than the group of isometries, however,  examples of manifolds with {\em essential conformal transformations} are relatively rare.
In fact, for Riemannian manifolds, Ferrand \cite{Lelong-Ferrand71} and Obata \cite{Obata71} showed that a compact Riemannian manifold with essential conformal transformations must be conformally diffeomorphic to the round sphere. More surprisingly, any non-compact Riemannian manifold with essential conformal transformations must be  conformally diffeomorphic to Euclidean  space, \cite{Ferrand96}. These results confirmed  the {\em Lichnerowicz conjecture}, \cite{Lichnerowicz64}. The conjecture can be extended to conformal structures of indefinite metrics, however already in Lorentzian signature it turns out to be false: there are many non compact  Lorentzian manifolds that are not conformally flat but with essential conformal transformations  \cite{Alekseevski85,Podoksenov89,Podoksenov92,KuhnelRademacher95,KuhnelRademacher97}, and Cahen-Wallach spaces are amongst them.
In Section~\ref{fixpoint-sec} we determine which conformal transformations of a conformally curved Cahen-Wallach space are essential. 
\begin{theorem}\label{intro-esstheo}
Let $\phi$ be a  homothety of a Cahen-Wallach space that is not an isometry. Then the following are equivalent:
\begin{enumerate}
\item $\phi$ is essential;
\item $\phi$ has a fixed point;
\item in equation (\ref{homoth}) for $\phi$ it is $\epsilon=-1$ or $c=0$. 
\end{enumerate}
In particular, when the Cahen-Wallach space is not conformally flat, then every essential conformal transformation is given by a homothety with the above properties.
\end{theorem}

Returning to the compact case in the Lichnerowicz conjecture for indefinite metrics, Frances \cite{Frances05} constructed examples of compact Lorentzian manifolds with essential conformal transformations that are not conformally diffeomorphic to the flat model of constant curvature, however, all of the examples constructed by Frances are conformally flat, i.e.~have vanishing Weyl tensor. This leads to the {\em generalised pseudo-Riemannian Lichnerowicz conjecture:} any compact pseudo-Riemannian manifold with essential conformal transformations must have vanishing Weyl tensor. Again Frances constructed counterexamples in all but Lorentzian signature \cite{frances12}, which leaves us with the {\em Lorentzian Lichnerowicz conjecture:} any compact Lorentzian manifolds with essential conformal transformations is conformally flat. This conjecture remains unproven in general until now, although substantial progress has been made \cite{FrancesMelnick10, Pecastaing17,Pecastaing18,Melnick21} and it has recently has been proved for compact real analytic manifolds that are simply connected \cite{MelnickPecastaing21} or of dimension three \cite{FrancesMelnick21}.

The counterexamples found by Frances \cite{frances12}  in signatures beyond the Lorentzian start with a locally symmetric space in signature $(2+p,2+q)$  which  admits a group of homotheties that acts with compact quotient and  centralises an essential homothety, which then descends to the compact quotient manifold. These examples are a motivation for our results in Section~\ref{centralsec}. Here we study whether for a given 
 essential homothety  $\phi$   of a Cahen-Wallach space there is  a group  $\Gamma$ of conformal transformations that acts properly discontinuously  and cocompactly and such that $\phi$ is in the centraliser  of  $\Gamma$.
In this case, $\phi$ would descend to an essential conformal transformation of the compact conformal manifold $M$.
 We will show however, that this is not possible.
				\begin{theorem}\label{thm_main-no-essential-quotient}
For				 a conformally curved Cahen-Wallach space,
			a group of conformal transformations that centralises an essential homothety cannot act properly discontinuously and cocompactly.
		\end{theorem}
This does not exclude the possibility of compact conformal quotients of Cahen-Wallach spaces with essential conformal transformation. We believe however, that no such quotient exists. 

The structure of the paper is as follows: in Section~\ref{prelimsec} we recall some basic notations and facts from conformal geometry including a short section about group actions; in Section~\ref{cwsec} we give a criterion for conformal flatness, describe the isometries, homotheties and conformal transformations of Cahen-Wallach spaces and prove Theorem~\ref{intro-esstheo}; Section~\ref{compact-sec} contains the proofs of the non-existence results in Theorems~\ref{intro-thm_no-homoth-quotients-of-imaginary-CW} and~\ref{thm_main-no-essential-quotient}. The article concludes with a few examples that illustrate the obstacles for constructing compact conformal quotients.

\subsection*{Acknowledgements} 
Most of the results in this paper are contained in second author's MPhil thesis \cite{stuart-thesis}, which was written under supervision of the first author. We would like to thank Michael~Eastwood for taking the role of co-supervisor and for helpful discussions and comments. The first author would also like to thank Vicente~Cort\'{e}s for inspiring  discussions on the topic and the University of Hamburg for its hospitality.

\section{Preliminaries on conformal geometry}
\label{prelimsec}

\subsection{Curvature conventions and conformal rescalings}
Let $(M,g)$ be a semi-Riemannian manifold of dimension $m$ and $\nabla$ the Levi-Civita connection.
Our convention for the curvature tensors are as follows: the Riemannian curvature as a $2$-form with values in $\so(TM,g)$, or equivalently as a $(1,3)$-tensor,  is defined as
\[
\RR(X,Y)=\left[\nabla_X,\nabla_Y\right]-\nabla_{[X,Y]},\]
and the $(0,4)$-curvature tensor
as
\[\RR(X,Y,Z,V)=g(\RR(X,Y)V,Z).\]
The Ricci-tensor is the trace of the $(1,3)$-curvature tensor
\[\Ric(Y,Z)=\tr( X\mapsto \RR(X,Y)Z ).\]
We denote the corresponding endomorphism also by $\Ric$ and 
and the scalar curvature $\scal$ as its trace.
Moreover, we define the Schouten tensor $\ro$ by
\[\Ric=(m-2)\ro +\tfrac{\scal}{2(m-1)}g,\]
and the $(0,4)$-Weyl tensor as 
\[\W(X,Y,Z,V)=
\RR(X,Y,Z,V)-g\owedge \ro,
 \]
where $\owedge$ is the {\em Kulkarni-Nomizu product} of two symmetric bilinear forms defined as 
\begin{eqnarray*}
A\owedge B(X,Y,Z,V)&=&
A(X,Z)B(Y,V)+ B(X,Z)A(Y,V)
\\
&&{}
- A(X,V)B(Y,Z)+ B(X,V)A(Y,Z),\end{eqnarray*}
and produces a $(0,4)$-tensor with the same symmetries as the Riemannian curvature tensor.
We define the $(1,3)$-Weyl tensor by \[g(\W(X,Y)Z,V)=W(X,Y,Z,V).\]  
We say that  $g$ is {\em Weyl-flat} if the Weyl tensor of $g$ vanishes.

If $\hg=\e^{2f}g$ is a conformally equivalent metric to $g$, where $f$ is a smooth function on $M$, then the Levi-Civita connection, the $(0,4)$-curvature tensor, and the Ricci and scalar curvature change as follows, see \cite[Section 1.J]{besse87},
\begin{equation}
\label{confchange}
\begin{array}{rcl}
\widehat{\nabla}_XY&=&\nabla_XY+\d f( X)Y+\d f ( Y)X-g(X,Y)\nabla f,\\[2mm]
\e^{-2f}\hR&=&\RR -g\owedge\left( \nabla \d f - (\d f)^2 +\tfrac{1}{2}g(\nabla f,\nabla f)g\right),
\\[2mm]
\hric&=&\Ric-(m-2)  \left( \nabla \d f - (\d f)^2 \right) +\left(\Delta f - (m-2) g(\nabla f,\nabla f)\right)g,
\\[2mm]
\e^{2f}\hscal&=&\scal +(m-1)\left( 2\Delta f -(m-2) g(\nabla f,\nabla f)\right),
\end{array}
\end{equation}
whereas the $(1,3)$-Weyl tensor is conformally invariant. Here
 $\nabla f$ is the gradient of $f$ and $\Delta f$ the Laplacian, both with respect to $g$. 
We observe the following useful relation.
\begin{lemma}\label{obslem}
If both $g$ and $\hg=\e^{2f}g$ have vanishing scalar curvature then 
\[
\e^{-2f}\hR =\RR+\tfrac{1}{m-2}g\owedge\left(\hric-\Ric\right).\]
\end{lemma}

 \subsection{Homotheties, conformal and essential conformal transformations}
 
A {\em conformal diffeomorphism}   between  semi-Riemannian manifolds $(M,g)$ and $(\hM,\hg)$   is a diffeomorphism $\phi:M\to \hM$ for which  there is a smooth function $f$ on $M$ such that 
\[\phi^*\hg=\e^{2f}g.\]
A conformal diffeomorphism for which $f$ is constant is called a {\em homothety}. We call a homothety  {\em strict} if it is not an isometry. We denote by $\Conf(M,g)$ the conformal transformations of $(M,g)$, i.e.~the conformal diffeomorphisms from $(M,g)$ to itself,  by $\Homoth (M,g)$ the homotheties of $(M,g)$, and by $\Iso(M,g)$ the isometries of $(M,g)$.

We say that $(M,g)$ is {\em conformally flat} if each point admits a local conformal diffeomorphism from a neighbourhood into a flat semi-Riemannian manifold. If there is a global conformal diffeomorphism from $M$ to a flat semi-Riemannian manifold, we say that $(M,g)$ {\em globally conformally flat}. All manifolds of dimension $2$ are conformally flat. In dimension $3$, $(M,g)$ is conformally flat if and only if the Cotton tensor $A$ of $g$ vanishes, which is defined as 
\[ 
A(X,Y,Z)=\nabla_Y\ro (Z,X) -\nabla_Z\ro(Y,X).
\] 
When $\dim(M)\ge 4$, $(M,g)$ is  conformally flat if and only if $g$  is Weyl-flat. 
Moreover, if two metrics $g$ and $\hat g$ on $M$ are {\em conformally equivalent}, i.e.~$\hg=\e^{2f}g$, then the identity transformation is a conformal diffeomorphism between $(M,g)$ and $(M,\hg)$. 

Let $(M,g)$ be a semi-Riemannian manifold. The map $\Homoth(M,g)\to \R$ that sends a homothety $\phi$ with $\phi^*g = \e^{2s}g$ to $s$ 
 is a group homomorphism with kernel $\Iso(M,g)$.
Hence we have that			\[\Homoth(M,g)=\Iso(M,g)\rtimes H,\]  
			for some subgroup $H$ of $\R$.
			Further, if for each $s\in\R$, there is a $h_s$ such that $h_s^*g=\e^{2s}g$, then $\Homoth(M,g)=\Iso(M,g)\rtimes\R$.
 An important result for our purposes is the following:
 \begin{theorem}[{Cahen \& Kerbrat \cite[Proposition 2.1]{CahenKerbrat82}}]\label{prop_cahen-kerbrat-weyl-curved-implies-no-conformals}
			Let $(M,g)$ be a connected semi-Riemannian manifold of dimension $m\geq 4$ with parallel Weyl tensor, $\nabla W=0$.
			Let $U\subset M$ be open and  $\phi:U\to \phi(U)\subset M$ be a conformal diffeomorphism. Then $\phi$ is a homothety or the Weyl tensor is identically zero.
 \end{theorem}
Since  $\nabla\RR=0$ implies that $\nabla\W=0$, we obtain:
 \begin{corollary}\label{cor_locally-symmetric-implies-Cahen-Kerbrat-result}
			Let $(M,g)$ be a connected, locally symmetric, semi-Riem\-an\-nian manifold of dimension $m\geq 4$.
			Let $U\subset M$ be open, and let $\phi:U\to \phi(U)\subset M$ be a conformal diffeomorphism. Then $\phi$ is a homothety, or the Weyl tensor is identically zero.
		\end{corollary}
	If $\phi$ is a homothety with $\phi^*g=\e^{2s}g$, the volume form $\nu$  of  a semi-Riemannian manifold satisfies 
	that $\phi^*\nu =\e^{ms}\nu$. Hence, 
		 compact semi-Riemannian manifolds cannot have any strict homotheties (see for example \cite{Alekseevski85}). This yields another result, which we could not find in the literature.
 \begin{corollary}\label{locsymcompconf}
			Let $(M,g)$ be a connected, compact and locally symmetric semi-Riem\-an\-nian manifold of dimension $m\geq 4$. Then the conformal group is equal to its isometry group.
			\end{corollary}
			
 A conformal transformation $\phi$ on a semi-Riemannian manifolds $(M,g)$ is called {\em essential} if there is no conformally equivalent metric on $M$ for which $\phi$ is an isometry. Similarly, the conformal group $\Conf(M,g)$ is called {\em essential} if there s no conformally equivalent metric $\hg$ on $M$ for which $\Conf(M,g)$ is contained in the isometries of $(M,\hg)$. Clearly, if $(M,g)$ admits an essential conformal transformation, then its conformal group is essential, however the converse, that an essential conformal group contains an essential conformal transformation,  is not obvious.
 For Riemannian conformal structures this implication follows from from the confirmed Lichnerowicz conjecture \cite{Lelong-Ferrand71,Obata71,Ferrand96}.

For homotheties, there is a  sufficient condition for being essential.
		\begin{proposition}\label{prop_finite-orbit-homoth-implies-essential}
			Let $(M,g)$ be a semi-Riemannian manifold and let $\phi$ be a strict homothety  with a finite orbit 
			point $p$, i.e.~ for some $k>0$, $\phi^k(p)=p$. Then $\phi$ is essential. In particular, a strict homothety with fixed point is essential.
		\end{proposition}
		\begin{proof}
			Let $\phi^*g = \e^{2s}g$ for $s\in\R$.
Assume $\phi^k(p)=p$ for $p\in M$ and that $\phi$ is not essential: let $f$ be a smooth function on $M$ such that $\phi$ is an isometry of $\e^{2f}g$.
			Then we evaluate at $p$,
			\begin{align*}
				(\e^{2f}g)|_p = (\phi^k)^*(\e^{2f}g)|_p = \e^{2f\circ \phi^k}(p)((\phi^k)^*g)|_p = \e^{2ks}\e^{2f(p)}g|_p = \e^{2ks}(\e^{2f}g)|_p.
			\end{align*}
			But  this implies $s=0$, and so $\phi$ is an isometry.
		\end{proof}
We will later see that for Cahen-Wallach spaces also the converse holds, so that the essential homotheties are {\em exactly} those with a fixed point.

\begin{remark}
In \cite[Theorem 2.1]{Alekseevski85} the converse of Proposition~\ref{prop_finite-orbit-homoth-implies-essential} is claimed for causal Lorentzian manifolds, i.e.~that every fixed point free homothety of a causal Lorentzian manifold is inessential. This relied on a claim in \cite{Alekseevskii72}, the proof of which has a gap. An example of an  essential homothety without fixed point on a causal Lorentzian manifold is easily constructed.
\end{remark}

\subsection{Properly discontinuous cocompact and conformal group actions}
	Let $\Gamma$ be a group of diffeomorphism acting on a smooth manifold $M$. The group action is \emph{properly discontinuous} if it satisfies the following two conditions
			\begin{itemize}
				\item[(PD1)] For each point $x\in M$, there is a neighbourhood $U$ of $x$ such that if $\gamma U$ meets $U$, i.e.~$\gamma U\cap U\neq\emptyset$ for $\gamma\in \Gamma$, then $\gamma=e$, where $e$ is the identity element. 
				\item[(PD2)] For all pairs of points $x,y\in M$ in different orbits, there are neighbourhoods $U$ of $x$, and $V$ of $y$, such that for all $\gamma\in \Gamma$, $\gamma U$ and $V$ are disjoint, i.e.~$\gamma U\cap V=\emptyset$.
			\end{itemize}
Clearly, (PD1) implies that $\Gamma$  acts freely, that is, its elements act without fixed points. If a group $\Gamma$ acts properly discontinuously on a manifolds $M$, then there is a unique smooth manifold structure on the orbit space $M/\Gamma$ and $\pi:M\to M/\Gamma$ is a covering map, see for example \cite[Proposition 7 in Chapter 7]{oneill83}.

If a group $\Gamma$ acts by diffeomorphisms on $M$ such that the orbit space $M/\Gamma$ is compact we say that $\Gamma$ acts {\em cocompactly}. 
If a group $\Gamma$ acts by diffeomorphisms on $M$ and admits a fundamental region with compact closure, then $M/\Gamma$ is compact.  The converse is not true in general in the sense that a $\Gamma$ acting on $M$  can have a compact orbit space but admits  fundamental domains with non-compact closures. 
 For metric spaces the converse holds if the fundamental region is assumed to be locally finite \cite[Chapter 6]{Ratcliffe06}. 
In \cite{FSApaper,stuart-thesis} we prove a converse that holds also for non-isometric (in the metric space sense) group actions, but requires the following, stronger assumption on the fundamental region (we state it here for smooth actions): if a group $\Gamma$ acts smoothly on a manifold $M$, then we call  a fundamental region $R$  {\em  finitely self adjacent} if there is a neighbourhood $U$ of the closure $\overline{R}$ of $R$, such that $\gamma(U)$ meets $U$ for only finitely many $\gamma\in \Gamma$. Note that a finitely self adjacent fundamental domain is necessarily locally finite, i.e.~$\{\gamma(\overline{R})\}_{\gamma\in \Gamma}$ is a locally finite family of sets.
From \cite{FSApaper,stuart-thesis} we obtain:
\begin{lemma}\label{FSAlemma}
Let $\Gamma$ be a group of diffeomorphisms acting on a manifold $M$ such that the topological space $M/\Gamma$ is compact. If $R$ is a finitely self adjacent fundamental region, then its closure must be compact.
\end{lemma}

Now let $(\tM,\tg)$ be a semi Riemannian manifold and $\Gamma$ a group that acts properly discontinuously on $\tM$. If $\Gamma$ is contained  isometry group of $(\tM,\tg)$, then the orbit space $M=\tM/\Gamma$ is equipped with a unique semi-Riemannian metric $g$ such that $\pi^*g=\tg$, where $\pi:\tM\to M$ is the covering map.  Similarly, when $\Gamma$ is a group of conformal transformation, the orbit space $M=\tM/\Gamma$ is equipped with a conformal structure $\mathbf{c}$ such that $\pi:\tM\to M$ is a conformal covering map, that is, $\pi$ is a covering map and for each $g\in \mathbf{c}$ there is a function $f\in C^\infty(\tM)$ such that $\pi^*g=\e^{2f}\tg$.  Note that the original metric $\tg$ on $\tM$ in general is not a lift of a metric in $\mathbf{c}$. This is only the case if $\Gamma$ is consists of isometries.
For more details and results, see \cite{stuart-thesis}.

\section{Conformal transformations of Cahen-Wallach spaces}
\label{cwsec}
\subsection{Conformal flatness of Cahen-Wallach space} 
%
M.~Cahen and N.~Wallach have shown in \cite{cahen-wallach70,Cahen98} that an indecomposable simply connected Lorentzian symmetric space either has constant curvature or is isometric to  a {\em Cahen-Wallach space}, which is defined as a Lorentzian manifold $(\R^{n+2},g_S)$ with $n\ge 1$ and 
$g_S$ is the metric in (\ref{cwmetric}) defined by a symmetric $(n\times n)$-matrix  $S=(S_{ij})$ with non-zero determinant. 
The condition that $S_{ij}$ is invertible is to ensure that $(M,g)$ is indecomposable.   If $S_{ij}$ is not invertible, then the metric $g_S$ is  a product of Euclidean space $\R^k$ and a Cahen-Wallach space of dimension $n+2-k$. Some of our results remain valid when  $S$ is not invertible, and we will point out when this is the case.  
 Clearly, if $S$ is the zero matrix, $g_S$ is just the Minkowski metric. 
 We call the metric $g_S$ a {\em Cahen-Wallach metric}, even when it is only defined on an open subset of $\R^{n+2}$.

We denote by $\Sigma$ the spectrum of the matrix $S$ and by $\Sigma_\pm$ the positive and negative eigenvalues.  A Cahen-Wallach space is of {\em real type} if $\Sigma=\Sigma_{+}$ and of {\em imaginary type} if $\Sigma=\Sigma_{-}$. Otherwise it is  of {\em mixed type}, see \cite{KathOlbrich15}.
 Two Cahen-Wallach spaces are isometric if and only if the corresponding matrices $S$ and $\hat S$  have the same spectrum with the same multiplicities  up to multiplication by a positive number, so that $\hat\Sigma =a\Sigma$, with $a>0$.

For the metric $g_S$, even when $S$ is degenerate, 
the vector field $\p_v=\frac{\p}{\p v}$, and consequently the one-form $\d t=g(\p_v,.)$, are  parallel and null. Moreover, 
\begin{equation}\label{cwnab}
\nabla \p_i =x^jS_{ij} \,\d t \otimes \p_v,\qquad \nabla \p_t= x^i S_{ij} \left( \d x^j\otimes \p_v - d t \otimes \p_j\right),
\end{equation}
where we use Einstein's summation convention.
By slightly abusing notation, we define the symmetric bilinear form $S=S_{ij}\d x^i\d x^j$
on $M$, so that the curvature of $g_S$ is given as
\begin{equation}\label{cwcurv}
\RR=-S\owedge (\d t)^2,\end{equation}
and consequently $\nabla \RR=0$, 
and the Ricci curvature is
\begin{equation}
\label{cwric}
\Ric=-\tr(S)  (\d t)^2,
\end{equation}
where $\tr(S) $ is the trace of the matrix $S$. Hence,  $g_S$ has vanishing scalar curvature and  Weyl tensor 
\begin{equation}\label{cwweyl}\W= -S\owedge  (\d t)^2 +\tfrac{\tr(S) }{n} g_S\owedge  (\d t)^2 = \left( \tfrac{\tr(S)}{n} \I - S\right) \owedge \d t^2,\end{equation}
where we define $\I=\delta_{ij}\d x^i \d x^j$ and use that \[g_S\owedge (\d t)^2= \I \owedge (\d t)^2.\] This yields the following result:
\begin{proposition}\label{Wflatprop}
The  metric $g_S$  is conformally flat if and only if $S$ is a scalar matrix.
\end{proposition}
Note that this result includes the case of dimension $3$, i.e.~when $n=1$: since the Ricci tensor is parallel, the Cotton tensor of a Cahen-Wallach metric always vanishes.
It also implies that in each dimension there are exactly two non-isometric Weyl-flat Cahen-Wallach spaces, namely those with $S=\pm\1$, where $\1$ is the identity matrix of $n$ dimensions. We denote their metrics by $g_\pm$. Since $g_\pm$ are  conformally flat, every local conformal transformation is given, via conjugation with the local conformal diffeomorphism to $\R^{1,n+1}$, by a local conformal transformation of Minkowski space, that is,  by the composition of a similarity and a local  inversion of  $\R^{1,n+1}$.

Due to Kuiper's result about the conformal development map \cite{MR31310}, see also \cite{KuhnelRademacher08} for a survey, a conformally flat Cahen-Wallach space $(\R^{n+2},g_\pm)$   embeds into the conformally flat model space, the Einstein universe $\S^1\times \S^{n+1}$ with conformal class defined by the product metric  $-\d \theta^2 + g_{\S^{n+1}}$, which has the conformal group $\mathbf{PO}(2,n+2)$, see \cite{frances08} for a survey. Hence, the Lie algebra of conformal vector fields of a conformally flat Cahen-Wallach space  $(\R^{n+2},g_\pm)$ has maximal dimension, that is  $\tfrac{1}{2}(n+4)(n+3)$. 
An explicit basis for the Lie algebra of conformal vector fields
was given in \cite[Proposition~4.4]{CahenKerbrat78}.

Since our focus is on the conformally curved case, we will not study the conformal group of $(\R^{n+2},g_\pm)$ further, we will only make a few more comments on the difference between the real and the imaginary case.
In terms of global rescaling to a flat metric we have:
\begin{proposition}
Let  $(\R^{n+2},g_\pm)$ be the conformally flat Cahen-Wallach spaces of dimension $n+2$. 
Then any conformal rescaling to a Ricci-flat metric is a rescaling to a flat metric and 
\begin{enumerate}
\item 
the metric $g_0=\e^{2 t } g_+$
on $\R^{n+2}$ is flat;
\item there is no global rescaling $f\in C^\infty(\R^{n+2})$ such that $\e^{2f}g_-$ is flat. 
\end{enumerate}
%
\end{proposition}
\begin{proof}
With $S=\epsilon\I$, with $\epsilon=\pm 1$,  the curvature of $g_\epsilon$ is 
\[\RR=-\epsilon \,\I \owedge (\d t)^2 = -\epsilon \, g_S\owedge  (\d t)^2.\]
Now assume that $f$ is a rescaling to a Ricci-flat metric $\hg=\e^{2f}g_\epsilon$.
Both metrics have vanishing scalar curvature, so  Lemma~\ref{obslem} and equations (\ref{cwcurv}) and (\ref{cwric}) yield
\[
\e^{-2f}\hR =\RR-\frac{1}{n}g_S\owedge \Ric=
\left( -\epsilon +\tfrac{\tr(S)}{n}\right) g_S\owedge (d t)^2 =0,
\]
as $\tr(S)=n\epsilon$. Hence, $\hg$ is not only Ricci-flat, but also flat.

By equations (\ref{confchange}), a conformal rescaling $f$ to a Ricci-flat metric
satisfies 
$2\Delta f= ng(\nabla f,\nabla f)$ and hence 
is a solution to
\begin{equation}
\label{ric0eq}
0= \epsilon  (\d t)^2+  \left( \nabla \d f - (\d f)^2 \right) +\tfrac{1}{2} g(\nabla f,\nabla f)g.\end{equation}
If $\epsilon=1$, then a solution to this equation is $f= t$. Indeed, $\d f =  \d t$, and hence $\nabla \d f=0$ and $g(\nabla f,\nabla f)=0$.

Now assume that  $\epsilon=-1$ and that $f\in C^\infty(M)$ is a global solution to equation (\ref{ric0eq}). We consider  the function $h(t)=f(t,0,\ldots, 0)$ which is defined on $\R$.
Evaluating the equation in the $\p_t$ direction and only along $(t,0,\ldots, 0)$ yields 
\[
0=
-1+{\p^2_t}f -x^i\p_if  -(\p_tf)^2 -\tfrac{\|\x\|^2}{2}g(\nabla f,\nabla f)
=
 -1+{\p^2_tf}  -(\p_tf)^2,
\]
 which shows that $h$ satisfies the ODE
\[\ddot{h}=\dot{h}^2+1.\]
Its derivative $y=\dot h\in C^\infty(\R)$ satisfies the first order separable equation 
\[\dot y=y^2+1.\]
Since $\int\frac{1}{y^2+1}\d y =\arctan(x)$ is bounded, the maximal domain of the solutions $y$ is also bounded  (see for example \cite[p.~9]{teschl12}), which contradicts the assumption.
\end{proof}

In the real type case this proposition shows that $g_+$ is  {\em globally} conformally equivalent to a flat Lorentzian metric $\hg$ on $\R^{n+2}$. We will see that this flat metric is in fact a {\em geodesically  incomplete}  Lorentzian metric. 
For this,  use (\ref{confchange}) and (\ref{cwnab}) to show that the vector fields  
\[\p_v,\quad Y_i:= \e^{-t}\left( \p_i+x^i\p_v\right),\quad Z:=\e^{-2t}\p_t -\e^{-t}x^kY_k=\e^{-2t}\left( \p_t-x^k\p_k- \|\x\|^2\p_v\right)
\]
on $\R^{n+2}$ are parallel for $\widehat{\nabla}$ and satisfy
\[
\hg (\p_v,\p_v)=\hg (\p_v,Y_i)=\hg (Y_i,Z)= \hg (Z,Z)=0,\quad \hg(\p_v,Z)=1,\quad \hg(Y_i,Y_j)=\delta_{ij}.\]
Observe now that the vector field  $Z$ is not complete. For example, its maximal integral curve through the origin is given as
\[\gamma(s) = \begin{pmatrix} t(s)=\frac12 \ln(2s+1) \\0\\0\end{pmatrix}.\]
Since $Z$ is parallel, this is also a maximal geodesic for $\hg$, which consequently is a geodesically  {\em  incomplete} flat Lorentzian metric on $\R^{n+2}$.

We will find an explicit conformal transformation between $(\R^{n+2},g_+)$ and an open set in Minkowski space $\R^{1,n+1}$.  Since $\p_v$, $Y_i$ and $Z$ are parallel, their metric duals are closed $1$-forms and we can  find a diffeomorphism 
\[\phi=\begin{pmatrix}
 u\\ y^i\\ z\end{pmatrix}\]
of  $\R^{n+2}$ by integrating the equations
\[
\begin{array}{rcccl}
\d u&=&\hg(\p_v,.,)& =& \e^{2t} \d t,
\\
\d y^i&=&\hg (Y_i,.) & =& \e^t\d x^i+\e^t x^i\d t,
\\
\d z &=& \hg(Z,.)& = &\d v- x^k \d x^k.
\end{array}
\]
A solution that yields a diffeomeorphism $\phi: \R^{n+2}\to \{u>0\}\subset \R^{1,n+1}$ is 
given by
\[
u=
\frac{\e^{2t}}{2},\quad 
y^i=
\e^t x^i
,\quad 
z=v-\tfrac{\|\x\|^2}{2}.
\]
 Hence we arrive at:
 \begin{proposition}\label{confmink}
 Let $(\R^{n+2},g_+)$ be the Weyl-flat Cahen-Wallach space of real type and let $(M,g_0)$ be the  Minkowski half space,
 \[M=\{(u,y^1,\ldots , y^n,z) \in \R^{n+2}\mid u>0\},\quad g_0=2\d u \d z +\delta_{ij} \d y^i\d y^j. \]
 Then $\phi$ defined by
\[
\R^{n+2}\ni \begin{pmatrix}
 t\\ \x\\ v\end{pmatrix}
\stackrel{\phi}{\longmapsto} 
\begin{pmatrix}u=
\frac{\e^{2t}}{2}\\[2mm]
\y=\e^t \x\\[2mm]
z=v-\frac{\|\x\|^2}{2}
\end{pmatrix}\in M
 \]
is a global conformal diffeomorphism between $(\R^{n+2},g_+)$ and $(M,g_0)$ with $
\phi^*g_0=\e^{2t}g_+$.
 \end{proposition}
 
 \begin{remark}
 Note that the inverse of $\phi$ is
 \[
\begin{pmatrix}
u\\
\y
\\
z
\end{pmatrix}\stackrel{\phi^{-1}}{\longmapsto} \begin{pmatrix}
t=\frac{1}{2} \ln (2u)
\\[1mm]
\x=\frac{\y}{\sqrt{2u}}
\\[2mm]
v=z+\frac{\|\veccy{y}\|^2}{4u}
\end{pmatrix},
\]
so that
$(\phi^{-1})^*g_+=\tfrac{1}{2u}g_0$.
Under conjugation by $\phi$, the isometries of  $(\R^{n+2},g_+)$ that are given by a translation in the $t$-component by $c$ (see next section),
\[
		\begin{pmatrix}t\\\veccy{x}\\v\end{pmatrix} \mapsto \begin{pmatrix} t+c \\\veccy{x} \\ v\end{pmatrix},
\]
are mapped to strict homotheties of $(M,g_0)$  of the form 
\[
		\begin{pmatrix}u\\\veccy{y}\\ z\end{pmatrix} \longmapsto \begin{pmatrix} \e^{2c} u \\ \e^c \veccy{y}\\ z\end{pmatrix},
\]
whereas the isometry of $g_+$, 
\[
		\begin{pmatrix}t\\\veccy{x}\\v\end{pmatrix} \longmapsto \begin{pmatrix} -t \\\veccy{x} \\ -v\end{pmatrix},
\]
is mapped to the non-homothetic conformal transformation  of $g_0$ given by
\[
		\begin{pmatrix}u\\\veccy{y}\\ z\end{pmatrix} \stackrel{\eta}{\longmapsto} \begin{pmatrix} \tfrac{1}{4 u} \\[2mm] \tfrac{ \veccy{y}}{2u} \\- z- \tfrac{\|\y\|^2}{2u}\end{pmatrix},
\]
that satisfies $\eta^*g_0=\frac{1}{4u^2}g_0$.
\end{remark}

In the imaginary case, only {\em local}  rescalings to a flat metric $\hg$ exist, for example, 
\[ \hg=\tfrac{1}{\cos^{2} (t) } \, g_-.\]
Similarly to the real case one, can show that the parallel vector fields of $\hg$ on $\{t\not=\frac{(2k+1) \pi}{2}\}$ are 
\[\partial_v, \quad Y_i=\cos (t)\, \partial_i +x^i\sin (t)\,\partial_v,\]
and \[ 
Z\ =\ \cos^2(t)\partial_t-x^i\sin(t) Y_i+\tfrac{\|x\|^2}{2} \partial_v\ =\ 
\cos^2(t)\partial_t-\tfrac{x^i}{2}\sin(2t) \p_i+\tfrac{\|x\|^2}{2} \cos (2t) \partial_v.
\]
Note that now the integral curves of $Z$ through the origin are given by
\[\gamma(s)=\begin{pmatrix}
\arctan(s)
\\
0
\\
0
\end{pmatrix}
\]
and hence defined for all $s$.
As before, for finding a diffeomeorphism $\phi=(u,y^i,z)$  we can integrate the equations
\[
\begin{array}{rcccl}
\d u&=&\hg(\p_v,.,)& =&\tfrac{1}{\cos^2(t)}  \d t,
\\
\d y^i&=&\hg (Y_i,.) & =& \tfrac{1}{\cos(t)}  \d x^i+\tfrac{\tan(t)}{\cos(t)}  x^i\d t,
\\
\d z &=& \hg(Z,.)& = &\d v- x^k \tan (t) \d x^k - \tfrac{\|x\|^2}{2\cos^2(t)} \d t
\end{array}
\]
and get a diffeomeorphism $\phi: \{ -\tfrac{\pi}{2}<t<\frac{\pi}{2} \}\to \R^{1,n+1}$ is 
given by
\[
u=
\tan (t) ,\quad 
y^i=
\tfrac{ x^i}{\cos (t) }
,\quad 
z=v-\tfrac{\|x\|^2}{2}\tan(t).
\]
Indeed, for the Minkowski metric $g_0=2\d u\d z +\sum_{i=1}^n(\d y^i)^2$ on $\R^{1,n+1}$, we have that 
\[ \phi^*g_0= \tfrac{1}{\cos^2(t) }g_-.\]

\subsection{Isometries, homotheties and conformal transformations}
\label{confcw-sec}
In this section we are going to determine the homothety group of a  Cahen-Wallach space, and consequently by virtue of Corollary~\ref{cor_locally-symmetric-implies-Cahen-Kerbrat-result}, in the conformally curved case also its conformal group. First we describe its isometry group, which is well-known since
 \cite{cahen-wallach70}, see also \cite{KathOlbrich15}.
 
 Let $(\R^{n+2}, g_S)$ in be a Cahen-Wallach space defined by the matrix $S$, which, for the moment, is not  assumed to be invertible.
 We denote
by   $\cent_{\O(n)}(S)$ the orthogonal matrices commuting with $S$ and by $V_S$ the $2n$-dimensional solution space of the ODE system $\ddot\beta=S\beta$, 
 \[V_S:=\{\beta:\R\to \R^n\mid \ddot\beta=S\beta\}.\]
It is straightforward to check that  the following  diffeomorphisms are  isometries of $g_S$,
			\begin{equation}\label{isom}
\psi=\psi_{c,\epsilon,b,\beta,A}:\begin{pmatrix}t\\\veccy{x}\\v\end{pmatrix} {\longmapsto} \begin{pmatrix}\epsilon\, t+c \\ A\veccy{x}+\beta(t) \\ \epsilon \left(v+b-\< \dot\beta(t),A\boldsymbol{x}+\frac12\beta(t)\>\right)\end{pmatrix},
			\end{equation}
			where $c\in \R$, $\epsilon\in\{\pm 1\}$, $A\in C_{\O(n)}(S)$,  $b\in \R$, $\beta\in V_S$ and $\<.,.\>$ is the standard Euclidean inner product on $\R^n$. Moreover, when $S$ is invertible,  using the fact that an isometry preserves the parallel null vector field $\p_v$,
						one can show that every isometry of $g_S$ is of this form (see \cite[Section 4.2]{stuart-thesis} for an explicit calculation).
			
In order to describe the group structure of the isometry group we will identify several subgroups of the isometry group and their relation to each other. All of these groups come with their natural action on $\R^{n+2}$ via formula (\ref{isom}). 	
		
First note that $ \R\times \{\pm1\}$ is a subgroup of the isometries with the  group structure of the   Euclidean group of $\R$, $\E(1)=\R\rtimes \Z_2$. We denote its elements by either $(c,\epsilon)$ or by $E_{c,\epsilon}$ when we refer to the  Euclidean motion $E_{c,\epsilon}(t)=\epsilon\,t+c$ of $\R$.			
Next, note that $\E(1)$ and $ \cent_{\O(n)}(S)$ commute with each other and that $\E(1)\times\cent_{\O(n)}(S)$ forms a subgroup of $\Iso(\R^{n+2},g_S)$, and we denote its elements by pairs $(E_{c,\epsilon}, A)$. 

Furthermore, also $\R\times V_S$, with its elements denoted by $(b,\beta)$,  forms a subgroup with group operation 
\[(b,\beta) +_\omega (\hat b, \hat \beta):= (b+\hat b +\omega (\beta ,\hat\beta) ,\beta+\hat\beta),\]
where $\omega$ is the symplectic form on $V_S$, defined by 
\[ \omega(\beta,\hat\beta) :=\tfrac{1}{2}\left(  \<\beta(0),\dot{\hat\beta}(0) \>- \<\dot\beta(0),\hat\beta(0)\>\right).\]
Note that the function $t\mapsto  \<\beta(t),\dot{\hat\beta}(t) \>- \<\dot\beta(t),\hat\beta(t)\>$ is actually constant, so in order to define $\omega$ we could have evaluated it at any other $t\not=0$. In particular, $\omega$ satisfies 
\begin{equation}
\label{shiftomega}
\omega(\beta,\hat\beta\circ E_{c,\epsilon} )=\epsilon\omega (\beta\circ E_{c,\epsilon}^{-1},\hat\beta),\quad\text{ for all $E_{c,\epsilon}\in \E(1)$,}
\end{equation}
which will turn out to be useful, as well as  \begin{equation}\label{Aomega}
\omega (A \beta, A\hat \beta)= \omega(\beta,\hat\beta),\quad\text{ for all $A\in \O(n)$.}\end{equation} This shows that $\R\times V_S$ has the group structure of the $(2n+1)$-dimensional Heisenberg group
\[\Hei_n:=\R\times_\omega V_S,\] 
which is the central extension of $V_S$ by $\R$. 

The subgroup $\Hei_n$ is normal in $\Iso(\R^{n+2},g_S)$. In fact, if $E_{c,\epsilon}\in \E(1)$ and $A\in \cent_{\O(n)}(S)$, we have for $(b,\beta)\in \Hei_n$ that
\[ (E_{c,\epsilon},A)(b,\beta)(E_{c,\epsilon},A)^{-1}=(E_{c,\epsilon},A)(b,\beta)(E_{\epsilon,-\epsilon c},A^\top)
=
\left( \epsilon b, A \beta\circ E_{-\epsilon c,\epsilon}\right)\in \Hei_n.\] 
Finally, any isometry $\psi$ as in (\ref{isom}) is a product of elements from $\Hei_n$ and $\E(1)\times \cent_{\O(n)}(S)$. Indeed, it is
\begin{equation}\label{actions}\psi =\psi_{c,\epsilon,b,\beta,A}= 
\underbrace{E_{\epsilon, c}}_{\in \E(1)}  \underbrace{(b,\beta)}_{\in \Hei_n}  \underbrace{A}_{\in \cent_{\O(n)}(S)}
=
\underbrace{\left( \epsilon b,  \beta\circ E_{-\epsilon c,\epsilon}\right)}_{\in \Hei_n}  \underbrace{(E_{c,\epsilon},A)}_{\in \E(1)\times\cent_{\O(n)}(S)}.
\end{equation}

The reader may have noticed that, in order to keep the notation brief, we use it quite flexibly: for example by $A$ we refer to $\psi_{0,1,0,0,A}$, 
 by $(E_{c,\epsilon},A)$ to $\psi_{c,\epsilon,0,0,A}$, by $(b,\beta)$ to $\Psi_{0,1,b,\beta,\1}$, etc., and the group product the is the composition when acting on $\R^{n+2}$.
Hence, we have arrived at the well known fact \cite{cahen-wallach70,KathOlbrich15}:
\begin{proposition}\label{isoprop}
The isometry group of a Cahen-Wallach space $(\R^{n+2},g_S)$ is isomorphic to the semidirect product
\[
\Hei_n\rtimes_\alpha \left(\E(1)\times C_{\O(n)}(S)\right),\]
where 
 $\Hei_n$ is the $(2n+1)$-dimensional Heisenberg group, $\E(1)=\R\rtimes \Z_2$ is the Euclidean group in one dimension, $\cent_{\O(n)}(S)$ is the centraliser of the matrix $S$ in $\O(n)$, 
and the homomorphism 
$\alpha:\E(1)\times C_{\O(n)}(S)\to\Aut(\Hei_n)$ is defined as 
\[\alpha_{(c,\epsilon,A)} (b,\beta) :=\left( \epsilon b, A \beta\circ E_{c,\epsilon}^{-1}\right).\]
The isomorphism maps $\psi_{c,\epsilon,b,\beta,A}$ in (\ref{isom}) to 
\[
\left( (\epsilon b,  \beta \circ E_{-\epsilon c,\epsilon} ) ,E_{c,\epsilon},A\right)\in \Hei_n\rtimes(\E(1)\times \cent_{\O(n)}(S)).
\]
\end{proposition}

To be very explicit, let us emphasise again that the  action of an element $\left( (b,\beta), (E_{c,\epsilon},A) \right) 
$ of $\Hei_n\rtimes_\alpha \left(\E(1)\times C_{\O(n)}(S)\right)$ on $\R^{n+2}$ is given via (\ref{isom})  as 
\[\left( (b,\beta), (E_{c,\epsilon},A) \right) (t,\x,v)= \psi_{0,1,b,\beta,\1}\left( \psi_{c,\epsilon,0,0,A} (t,\x,v)\right).\]
Moreover, the 
  explicit formula for the group product in $\Hei_n\rtimes_\alpha \left(\E(1)\times C_{\O(n)}(S)\right)$ is
\begin{eqnarray}
\nonumber
\lefteqn{
\Big((b,\beta), (E_{c,\epsilon} ,A)\Big) \left((\hat b,\hat\beta), (E_{\hat c,\hat\epsilon} ,\hat A)\right)}\\
&
=&
\nonumber
\left( \left( (b,\beta)+_\omega\alpha_{c,\epsilon,A}  (\hat b,\hat \beta)\right),
 \left( E_{c,\epsilon}\circ  E_{\hat c,\hat \epsilon}, A\hat A\right)\right)
 \\
 &=&
 \label{product}
 \left( \left( (b+\epsilon \hat b + 
 \omega(\beta, A\hat \beta\circ E_{-\epsilon c,\epsilon}) , \beta+ A\hat \beta\circ E_{-\epsilon c,\epsilon}\right) ,( E_{c+
 \epsilon \hat c, \epsilon\hat \epsilon}, A\hat A)\right),
 \end{eqnarray}
 where $(b,\beta)$ and $(\hat b,\hat\beta)$ are elements from $\Hei_n$ and $(E_{c,\epsilon} ,A)$ and $(E_{\hat c,\hat\epsilon} ,\hat A)$ from $\E(1)\times C_{\O(n)}(S)$. 
The formula for the inverse is
  \begin{equation}\label{inverse}
 \Big((b,\beta), (E_{c,\epsilon} ,A)\Big)^{-1}= \Big((-\epsilon b,-A^\top\beta\circ E_{c,\epsilon}), (E_{-\epsilon c,\epsilon} ,A^\top)\Big).\end{equation}

\begin{remark}
We should also point out that if $S$ is not invertible, then the isometry group of $g_S$ contains the group $\Hei_n\rtimes_\alpha \left(\E(1)\times \cent_{\O(n)}(S)\right)$, but in general is larger, for example when $S=0$, in which case $g_S$ is the Minkowski metric. 

When $S$ is not zero but has a  kernel of dimension $k\ge 1$, the metric $g_S$ is isometric to a product of an indecomposable Cahen-Wallach space of dimension $n-k+2$ and Euclidean space of dimension $k$. Interestingly, since it contains $\Hei_n\rtimes_\alpha \left(\E(1)\times \cent_{\O(n)}(S)\right)$, the isometry group is larger than the product of the isometry groups of both manifolds,
whose dimension is  
\[\dim(\cent_{\O(n-k)}(S)) + 2(n+1) +  \frac{1}{2} k(k-3).\] 
On the other hand, since the centraliser of $S$ in $\O(n)$ is $\cent_{\O(n-k)}(S)\times \O(k)$, the dimension of $\Hei_n\rtimes_\alpha \left(\E(1)\times \cent_{\O(n)}(S)\right)$
is \[\dim(\cent_{\O(n-k)}(S)) + 2(n+1) +  \frac{1}{2} k(k-1).\] 
Moreover, observe that if
 $(\R^{n+2},g_S)$ is conformally flat, then  \[\Iso(\R^{n+2},g_S)=
\Hei_n\rtimes_\alpha \left(\E(1)\times \O(n)\right),\]
and hence 
 the dimension of the isometry group is reduced by $n+1$ from the  dimension of the isometry group of Minkowski space $\R^{1,n+1}$, which is $\tfrac{1}{2}(n+2)(n+3)$. 
 \end{remark}
\begin{remark}
The transvection group within the isometry group is the solvable group
$\Hei\ltimes_\alpha \R$, where $\R$ are the translations in $\E(1)$.
The stabiliser in $\Iso(\R^{n+2},g_S)$ of the origin is given as
\[L_S\rtimes_\alpha (\Z_2\times \cent_{\O(n)}(S)),\]
where 
\[L_S:=\{\beta\in V_S\mid \beta(0)=0\}\subset \Hei_n\]
is a Lagrangian subspace  in $V_S$ and hence a subgroup of $\Hei_n$. 
Similarly the stabiliser in the transvections is the abelian group  $L_S$ and we have
\[ (\R^{n+2},g_S)= (\Hei\ltimes_\alpha \R)/L_S.\]
\end{remark}
Now we turn to the homotheties of $(\R^{n+2}, g_S)$.  Clearly, for each $s\in \R$  the linear map given by the matrix $h_s:=\diag (1,\e^s,\ldots , \e^s, \e^{2s} )$,
\begin{equation}
\label{hs}
\begin{pmatrix}
t\\ \x \\ v\end{pmatrix}\stackrel{h_s}{\longmapsto}
\begin{pmatrix}
t\\ \e^s\x \\ \e^{2s}v\end{pmatrix}\end{equation}
 is a homothety of $g_S$. We call $h_s$ a {\em pure homothety}.  The pure homotheties are a subgroup in the homotheties which we denote by $\R$. 
The isometries are normal in the homotheties and we have that 
\begin{equation}
\label{hprod}
h_s\left((b,\beta)\cdot(c,\epsilon, A)\right) h_s^{-1}= (\e^{2s} b,\e^s \beta)\cdot (c,\epsilon, A) .
\end{equation}
In particular, the pure homotheties commute with $\E(1)\times \cent_{\O(n)}(S)$. 
This yields
\begin{proposition}\label{homoprop}
The homothety group of a Cahen-Wallach space is isomorphic to 
 \[ \Hei_n\rtimes_\varphi \left(\E(1)\times \cent_{\O(n)}(S)\times \R\right), \]
where $\vf:\E(1)\times \cent_{\O(n)}(S)\times \R\to \Aut(\Hei_n)$ is defined as 
\[\vf{(c,\epsilon,A,s)} (b,\beta) :=\left( \epsilon \e^{2s} b, \e^sA \beta\circ E_{c,\epsilon}^{-1}\right).\]
\end{proposition}
In the following we will denote 
\begin{equation}\label{defHS}
H_{S} := \Hei_n\rtimes_\vf \left(\E(1)\times \cent_{\O(n)}(S) \times \R\right),
\end{equation}
and identify it with the 
the homothety group of the Cahen-Wallach space $(\R^{n+2},g_S)$. From the proposition it follows that there is a surjective group homomorphism
\[H_S\longrightarrow H_S/\Hei_n \simeq \E(1)\times \cent_{\O(n)}(S)\times \R,\]
and for a homothety $\phi$ we denote the image under this projection by \[(E_\phi,A_\phi,s_\phi)=(c_\phi,\epsilon_\phi,A_\phi,s_\phi)\in \E(1)\times \cent_{\O(n)}(S)\times \R
=
(\R\ltimes \Z_2)  \times \cent_{\O(n)}(S)\times \R.\]
%
%

\begin{remark}
 It was already noted in \cite{Alekseevski85}, see also \cite{Podoksenov89,Podoksenov92}, that the diffeomorphism $h_s$ in (\ref{hs})
is a homothety for many Lorentzian metrics, namely those of the form
\[
2 \d t  \left( \d v + P_{ij}(t) x^i\d x^j + \left(Q_{ij}(t) x^ix^j+ R(t) v\right)\d t\right) +\delta_{ij}\d x^i \d x^j,\]
including the so-called pp-waves, of which the Cahen-Wallach metrics are a special case.\end{remark}

In the non Weyl-flat case  with $n\ge 2$,  the conformal group of $(\R^{n+2},g_S)$,  reduces to the homotheties by Corollary~\ref{cor_locally-symmetric-implies-Cahen-Kerbrat-result}.  
\begin{corollary}\label{curvedconf}
Let $(\R^{n+2},g_S)$ be a Cahen-Wallach space of dimension $n+2\ge 4$ such that $S$  has at least two different eigenvalues, i.e.~$g_S$ is not Weyl-flat.  If $\phi:U\to\phi(U)$ is a conformal transformation on an open set $U$,  then $\phi$ is a homothety. In particular,
\[\Conf(\R^{n+2},g_S)=H_S= \Hei_n\rtimes_\vf \left(\E(1)\times \cent_{\O(n)}(S)\times\R  \right) .\]
\end{corollary}

\subsection{Fixpoints and essential homotheties}
\label{fixpoint-sec}

	In this section, we will prove Theorem~\ref{intro-esstheo}. For this we detail various sufficient conditions for a homothety of a Cahen-Wallach space to have a fixed point and hence be essential. Using this, we will show the converse of Proposition~\ref{prop_finite-orbit-homoth-implies-essential}.
Throughout  this section, we denote by
$\phi=\phi_{c,\epsilon,b,\beta,A,s}$ the homothety
	\begin{equation}\label{homoth1}
\phi=\phi_{c,\epsilon,b,\beta,A,s}\ :\ \begin{pmatrix}t\\\veccy{x}\\v\end{pmatrix} {\longmapsto} \begin{pmatrix}\epsilon\, t+c \\ \e^sA\veccy{x}+\beta(t) \\ \epsilon \left(\e^{2s}v+b-\< \dot\beta(t),A\boldsymbol{x}+\frac12\beta(t)\>\right)\end{pmatrix},
			\end{equation}
 and by $E_\phi$ the Euclidean motion $E_{c,\epsilon}$ that maps $t$ to $\epsilon\,t+c$.
		\begin{proposition}\label{prop_homothetic-fixed-points}
A strict homothety $\phi$ of
			$(\R^{n+2},g_S)$ has a fixed point if and only if the Euclidean motion $E_\phi$ of $\R$  has a fixed point, i.e.~if and only if  $\epsilon=-1$ or $c=0$.
		\end{proposition}
		\begin{proof}
		Recall that  $E_\phi$ is defined by $E_\phi (t)=\epsilon\, t+c$. If $E_\phi$ has no fixed point, then $\phi$  cannot fix any point.
		
			Conversely, let $t$ be a fixed point of $E_\phi$. Then one can verify that
			$$\begin{pmatrix}t\\-(\e^sA-\1)^{-1}\beta(t)\\ -(\e^{2s}a-1)^{-1}a\left(b-\<\dot\beta(t),-\e^sA(\e^sA-I_n)^{-1}\beta(t)+\frac12\beta(t)\>\right)\end{pmatrix}$$
			is a fixed point of $\phi$. The assumption that $\phi$ is a strict homothety is crucial for the inverses of  $(\e^{2s}\epsilon-1)$ and $(\e^sA-\1)$ to exist.
		\end{proof}
		\begin{lemma}\label{lem_isometry-a=-1-fixed-points}
			Let $\phi$ be an isometry of $(\R^{n+2},g_S)$ with $\epsilon=-1$.
			Then $\phi$ fixes a point if and only if $\veccy x \mapsto A\veccy x +\beta(\tfrac{c}{2})$ fixes a point.
		\end{lemma}
		\begin{proof}
			Note that $\tfrac{c}{2}$ is the unique fixed point of $E_\phi=E_{c,-1}:t\mapsto -t+c$, so if $A\veccy x+\beta(\tfrac{c}{2})$ does not have a fixed point, then $\phi$ cannot fix any point.

			Conversely, let $\veccy y=A\veccy y+\beta(\tfrac{c}{2})$. Then one can check that
			$$\begin{pmatrix}\tfrac{c}{2}\\\veccy y\\-\frac12\left(b-\<\dot\beta(\tfrac{c}{2}),A\veccy y+\frac12\beta(\tfrac{c}{2})\>\right)\end{pmatrix}$$
			is a fixed point of $\phi$.
		\end{proof}
		\begin{proposition}\label{prop_torsion-implies-fixed-point}
			Let $\phi$ be a homothety of $(\R^{n+2},g_S)$ with $\phi^k=\id$ for some $k>0$. Then $\phi$ fixes a point.
		\end{proposition}
		\begin{proof}
			We construct a point $y$ fixed by $\phi$.
			We start with the $t$ component of $y$: If $\epsilon=1$, then $t \circ \phi^k = t+kc$. This can only fix a point for $k>0$ if $c=0$, so we conclude either $c=0$ or $\epsilon=-1$. In either case
			$$t(y):=\tfrac{c}{2}$$
			is a fixed point of $t\mapsto \epsilon\, t+c$. From this we also get by Proposition~\ref{prop_homothetic-fixed-points}, that $\phi$ either has a fixed point or is an isometry. So we assume that $\phi$ is an isometry. Define $\beta_0 := \beta(\tfrac{c}{2})$ and $\dot\beta_0 := \dot\beta(\tfrac{c}{2})$.

			Next we consider the $\veccy x$ component: since $\phi^k(\tfrac{c}{2},0,0)=(\tfrac{c}{2},0,0)$ we get that the Euclidean motion $E_{\beta_0,A}(\veccy x) = A\veccy x + \beta_0$ satisfies $E^k_{\beta_0,A} \equiv \Id$. In particular, \[E^k_{\beta_0,A}(0) = \sum_{i=0}^{k-1}A^i\beta_0 = 0.\]
			In general any euclidean motion $E$ satisfying $E^k(\veccy x) = \veccy x$ fixes a point. This fixed point is given by the centre of mass (in our case, $\veccy x=0$),
			$$\veccy y := \frac1k \sum_{i=1}^k E^i_{\beta_0,A}(0).$$
			Now when $\epsilon=-1$,  Lemma~\ref{lem_isometry-a=-1-fixed-points} gives us a fixed point of $\phi$.
			When $\epsilon=1$, to have a fixed point we require that $b-\<\dot\beta_0,A\veccy y+\frac12\beta_0\> = 0$. But since $\phi(\frac{c}{2} , \veccy y, .) = (\frac{c}{2}, \veccy y,\hdots)$, we get
			$$\phi^k\begin{pmatrix}\frac{c}{2}\\\veccy y\\0\end{pmatrix} = \begin{pmatrix}\frac{c}{2}\\\veccy y\\k(b-\<\dot\beta_0,A\veccy y+\frac12\beta_0\>)\end{pmatrix} = \begin{pmatrix}\frac{c}{2}\\\veccy y\\0\end{pmatrix}.$$
			Hence $b-\<\dot\beta_0,A\veccy y+\frac12\beta_0\>=0$, and so any choice of $v(y)$, for example $v(y):= 0$ makes $(\frac{c}{2},\veccy y,0)$ a fixed point of $\phi$.
		\end{proof}

Now we give a characterisation of essential homotheties of Cahen-Wallach spaces. The non-trivial part of the proof is a special case of the results in 
 \cite{FSApaper,stuart-thesis}, however we will present it here for the sake of completeness.

\begin{theorem}\label{thm_CW-essential-iff-fixed-point}
			A strict homothety $\phi$ of a Cahen-Wallach space is essential if and only if it fixes a point.
		\end{theorem}
		\begin{proof}
			First, if $\phi$ has a fixed point, then by Proposition~\ref{prop_finite-orbit-homoth-implies-essential}, $\phi$ is essential. For the converse, assume that 
 $\phi$ is essential but without fixed point.
			Then by Proposition~\ref{prop_homothetic-fixed-points}, $t\circ \phi(x) = t(x)+c$ for some  $c>0$. We will now construct a function $f$ such that $\phi$ is an ismometry for the metric $\e^{2f}g$.
			
The group $\langle\phi\rangle$ admits a fundamental domain $D=\R^{n+1}\times(0,c)$. Let $h:\R\to \R_{>0}$ be a smooth function of $t$ such that $h|_{[0,c]}\equiv 1$ and with support in $(-\tfrac{c}{4},\tfrac{5c}{4} )$. Then define functions $\{h_k\}_{k\in \Z}$ on $\R^{n+2}$ by
\[h_0(t,\x,v)=h(t),\qquad h_{k}:=h_0\circ \phi^{-k},\]
so that $h_k=h_{k+1}\circ\phi$.
Since $\{\mathrm{supp}(h_k)\}_{k\in \Z}$ is locally finite, the function $\sum_{k\in \Z}h_k$ is well defined. Since $D$ is a fundamental domain and $h_k\ge 0$, the function $\sum_{k\in \Z}h_k$ has no zeros. Hence,
\[f_k:=\frac{h_k}{\sum_{k\in \Z}h_k} \]
is a partition of unity on $\R^{n+2}$ that satisfies $f_{k}=f_{k+1}\circ\phi$. If $\phi^*g=\e^{2s}g$, we set
\[f:=-s\sum_{k\in \Z}k f_k,\]
which yields that $\phi$ is an isometry for $\e^{2f}g$ as 
\[
f\circ \phi= -s\sum_{k\in \Z}k f_{k-1}=f -s\sum_{k\in \Z}f_k=f-s.\]
Therefore, $\phi$ is not essential and we arrive at a contradiction, so that $\phi$ must have  a fixed point.
		\end{proof}

\section{Non-existence results for  compact conformal quotients of Cahen-Wallach spaces}
\label{compact-sec}
\subsection{Conformal compact quotients of imaginary type}
In this section we show that conformal compact quotients of Cahen-Wallach spaces of imaginary type must be isometric quotients. 

We start with some technical results about cocompact conformal group actions of Cahen-Wallach spaces. 
Recall the definition of $H_S$ in (\ref{defHS}), 
which is the group of homotheties of the Cahen-Wallach space $(\R^{n+2},g_S)$, i.e.~when $S$ is invertible. The technical statements that will lead up to Theorem~\ref{thm_no-homoth-quotients-of-imaginary-CW} however will not require that $S$ is invertible.
First we show that cyclic groups of homotheties cannot act cocompactly.

	\begin{lemma}\label{lem_cyclic-isnt-cocompact}
			If $\gamma\in H_S$, then
			 $\<\gamma\>$ does not act cocompactly.
		\end{lemma}
		\begin{proof}  
Assume that $\< \gamma \>$  acts cocompactly. Then $\Gamma:=\<\gamma^2\> $ acts cocompactly and we have that $t\circ\gamma^2 =t+c$, where $t$ is the coordinate 
\begin{align*}
				t:\R^{n+2} \to \R,\qquad (t,\x, v)\mapsto t.
			\end{align*}
If $c=0$, then 
the smooth map $t$ 
is invariant under $\Gamma$ and descends to a smooth, surjective map  $ f:\R^{n+2}/\Gamma \to \R$.
This contradicts the compactness of $\R^{n+2}/\Gamma$.

If  $c\neq 0$, 
then $D=\R^{n+1}\times(0,c)$ is a finitely self adjacent fundamental domain, see
\cite{FSApaper,stuart-thesis}, and hence, by Lemma~\ref{FSAlemma}, $M/\<\gamma\>$ cannot be compact. 
		\end{proof}

In regards to a group of homotheties acting properly discontinuous, we obtain from Proposition~\ref{prop_homothetic-fixed-points} the following corollary.

\begin{corollary}\label{cor:fp}
If $\Gamma$ is a group of homotheties acting properly discontinuously, then every $\gamma \in \Gamma \setminus \Iso$ must have $\epsilon =1$ and $c\not=0$. 
\end{corollary}

Next, we find an obstruction for a homothety group acting properly discontinuously and cocomapactly. Recall that  $\Sigma_{+}$ denotes the set of positive eigenvalues of $S$. 
		\begin{proposition}\label{prop_good-homoth-quotient-implies-exponential-beta}
			Let $\Gamma\subset H_S$ be a subgroup 
that			acts  cocompactly. 
If  $\Gamma$ contains a strict homothety \[\gamma=(c_\gamma,A_\gamma,s_\gamma)\in \R\times \cent_{\O(n)}(S) \times \R\] with the property
\begin{equation}\label{lcs}
\text{
 $ s^2_\gamma -\lambda^2_ic_\gamma^2>0\quad$
for all $\lambda_i^2\in \Sigma_+$},\end{equation}
then $\Gamma$ cannot act properly discontinuously.
		\end{proposition}
		\begin{proof} Since $\Gamma$ acts cocompactly, by Lemma~\ref{lem_cyclic-isnt-cocompact} it cannot be cyclic, so there is a homothety  $\phi\in \Gamma\setminus \<\gamma\>$, which we fix. 
		Without loss of generality, we can assume that $\phi$ is a strict homothety, otherwise we multiply $\phi$ by $\gamma$. 
		Let $\phi=\phi_{c,\epsilon,b,\beta,A,s}$ be a homothety as in (\ref{homoth1}).
			For a proof by contradiction,  assume that $\Gamma$ acts properly discontinuously. 
This implies that  $\epsilon=1$ and $c\not=0$, as otherwise, 
		by Proposition~\ref{prop_homothetic-fixed-points}, 
 $\phi$ would have a fixed point and $\Gamma$ could not act properly discontinuous. 

Furthermore, if $\beta$, $b$ are both zero, then for any sequence of rational numbers $p_i/q_i\to c/c_\gamma$, we have that $\gamma^{p_i}\phi^{-q_i}(0)\to 0$, contradicting PD1. Hence at least one of  $b$ or $\beta$ is non-zero.

The assumption that
\[ 0<s^2_\gamma -\lambda^2_ic^2_\gamma =(s_\gamma+\lambda_i c_\gamma)(s_\gamma-\lambda_i c_\gamma)\]
means that $(s_\gamma+\lambda c_\gamma)$ and $(s_\gamma-\lambda c_\gamma)$ have the same sign. Without loss of generality, we assume that both are positive (if both are negative, we use $\gamma^{-1}$ in what follows). 
			Now consider the sequence
			\[y_k=\gamma^{-k}\phi\gamma^k(0) = \begin{pmatrix}c  \\[1mm]\e^{-ks_\gamma}(A^\top_\gamma)^k \beta(kc_\gamma)\\[1mm] \e^{-2ks_\gamma}\left(b-\<\dot\beta(kc_\gamma), \frac12\beta(kc_\gamma)\>\right)\end{pmatrix}.
			\]
Let $\beta^i:\R\to\R$, $i=1, \ldots , n$ be the components of $\beta\in V_S$.
	For those $i$ for which the eigenvalues of $S$ are negative, the sequence $\beta^j(kc_\gamma)$ is bounded and hence 	$\e^{-ks_\gamma}	\beta^j(kc_\gamma)$ converges to zero when $k$ goes to infinity.
	
	For those $j$ that correspond to the kernel of $S$, $\beta^i$ is linear and again $\e^{-ks_\gamma}	\beta^i(kc_\gamma)$ converges to zero when $k$ goes to infinity.

For those $i$ for which the eigenvalues $\lambda_i^2$ of $S$ is positive we have
\begin{equation}\label{hyps}
\beta^i(kc_\gamma)=b^i\cosh(\lambda_i kc_\gamma )+ c^i\sinh (\lambda_i k c_\gamma)= \tfrac{b^i+c^i}{2}\e^{\lambda_i kc_\gamma}+ \tfrac{b^i-c^i}{2} \e^{-\lambda_i kc_\gamma},\end{equation}
for some constants $b^i$ and $c^i$, not both zero.
By the assumption that both $(s_\gamma+\lambda_i c_\gamma)$ and $(s_\gamma-\lambda_i c_\gamma)$ are positive, we have that 
\[
\e^{-ks_\gamma} \beta^i(kc_\gamma)= \tfrac{b^i+c^i}{2}\e^{-k(s_\gamma- \lambda_i c_\gamma)}+ \tfrac{b^i-c^i}{2} \e^{-k(s_\gamma+\lambda_i c_\gamma)}\]
(no summation over $i$), is a non-constant sequence that converges to zero when $k\to\infty$. Since $A_\gamma$ is an orthogonal matrix this implies that $y_k^i$ is a non-constant sequence converging to zero for $k\to \infty$.

Finally for  $v_k=v(y_k)=\e^{-2ks_\gamma}(b-\<\dot\beta(kc_\gamma), \frac12\beta(kc_\gamma)\>$, using the formula (\ref{hyps}) for the $\beta^i$s, we have that \[\lim_{k\to\infty} v_k =\frac14 \sum_{i=1}^p \lim_{k\to \infty} \left( \e^{-2k(s_\gamma-\lambda_ic_\gamma)}- \e^{-2k(s_\gamma+\lambda_ic_\gamma)}\right)=0,\]
since both $(s_\gamma\pm\lambda c_\gamma)$  are positive.

Hence, $y_k=\gamma^{-k}\phi\gamma^k(0)$ is a non-constant sequence  in the $\Gamma$-orbit of $0$ that converges to $(c ,0,0)$. This contradicts the assumption that $\Gamma$ acts properly discontinuously.
		\end{proof}
\begin{remark}
	The assumption of  Proposition~\ref{prop_good-homoth-quotient-implies-exponential-beta} can be formulated as 
\[\tfrac{s^2}{c^2}> \lambda_{\max}^2,
\]
where $\lambda_{\max}^2$ is the largest positive eigenvalue of $S$.
Note  that we have not assumed that $S$ is non-degenerate. Also, this condition is invariant in the isometry class of $g_{S}$. Indeed, for $a\in R_{>0}$ and $\alpha(t,x,v)=(at,x, a^{-1}v)$, we get that $\alpha^*g_S=g_{a^2 S}$ and the homotheties of $g_{a^2S}$ are given by conjugation with $\alpha$, i.e. by $\hat \gamma= \alpha^{-1}\circ\gamma\circ\alpha$. Then $\hat c$ of $\hat\gamma$ is given by $\frac{c_\gamma}{a}$, $s$ is not changed, and the largest eigenvalue of $a^2S$ is $a^2\lambda_{\max}$, so that 
\[\left(\frac{\hat s}{\hat c}\right)^2= a^2\left(\frac{ s}{ c}\right)^2>a^2\lambda_{\max}.\]

The proof also shows that if $\Gamma$ contains a $\gamma=(c_\gamma,A_\gamma,s_\gamma)\in \R\times C_{\O(n)}(S)\times \R$
and acts cocompactly and properly discontinuously, then  for all  $\phi\in \Gamma\setminus \<\gamma\>$ the corresponding  $\beta$ must have at least one exponentially growing component, that is, $\beta$ must have a non-vanishing component in an eigenspace for a positive eigenvalue.
\end{remark}

%
%

As the next step we show that under certain conditions, every strict homothety is conjugated to a homothety $\gamma\in \R\times C_{\O(n)}(S)\times \R$.
This requires the following technical result.

\begin{lemma}\label{lem_beta-solving-conjugation-equation}
			Let $\hat\beta\in V_S$, $A\in C_{\O(n)}(S)$, $s\in\R$, $s\not=0$, and $ c\in \R$. Unless $(\frac{s}{c})^{2}$ is  an eigenvalue of $S$, 
 there is a $\beta\in V_S$ such that
			\begin{equation}\label{betaeq}
 \e^{ s}  A \,\beta\circ \sigma_{c}-\beta =\hat\beta,
\end{equation}
where $\sigma_c=E_{c,1}\in\E(1)$ is the shift by $c$ in $\R$, $\sigma_c(t)=t+c$.
		\end{lemma}
		\begin{proof}
		Since $A$ commutes with $S$, it preserves the eigenspaces of $S$ and it suffices to determine $\beta$ on each eigenspace separately. We abuse notation by denoting by $\beta$ and $\hat \beta$ their component on each  eigenspace and by $n$ the dimension of an eigenspace.
		Since $\beta\in V_S$ and $\hat\beta\in V_S$, both are determined by their initial values $\beta_0:=\beta(0)$, $\hat\beta_0:=\hat\beta(0)$ and  derivatives $\beta_1:=\beta'(0)$ and $\hat\beta_1:=\hat\beta'(0)$ we have to show that we can choose $\beta_0$ and $\beta_1$ such that the corresponding solution $\beta$ satisfies (\ref{betaeq}).   First we consider the case of a negative eigenvalue $-\mu^2$. A solution $\beta $ is given as 
		\[
		\beta(t)=\beta_0\cos(\mu t)+\tfrac{\beta_1}{\mu}\sin (\mu t).
		\]
Hence we are looking for initial condition $\beta_0$ and $\beta_1$ satisfying the linear system of $2n$ equations
\begin{equation}\label{linsystrig}
\underbrace{\begin{pmatrix}
\mu ( \e^s \cos(\mu c) A -\1) &  \e^s \sin(\mu c)  A
\\[1mm]
-\mu \e^s  \sin(\mu c) A &  \e^s \cos(\mu c) A -\1
\end{pmatrix}}_{=:M}
\begin{pmatrix}\beta_0\\
\beta_1\end{pmatrix}
=
\begin{pmatrix}\mu\hat\beta_0\\
\hat\beta_1\end{pmatrix}
\end{equation}
This system has a solution for any given right-hand-side if the matrix $M$ is invertible.
			Since the bottom two blocks commute, we get the determinant 
\[
				\det(M) 
				= \mu^n \det\Big(\e^{2s}A^2-2 \e^s\cos(\mu c)A+\1 \Big),
\]
see \cite{Silvester00} for details.
			We know that $A$ is orthogonal and hence is diagonalisable over $\C$ with complex eigenvalues $z_1,\ldots , z_n$ which lie on the unit circle, $|z_i|=1$.  So we compute this determinant by  diagonalisation to obtain
			\begin{equation*}
				\det(M) = \prod_{i=1}^n\left( \e^{2s}z_i^2-2\e^s\cos(\mu c)z_i + 1\right).
			\end{equation*}
			Hence, the determinant of $M$ can only be zero if one of the $z_i$'s is a root of 
\[ \e^{2s}z^2-2\e^s\cos(\mu c)z+1.\]
However, the roots of this quadratic polynomial are given by the complex numbers $z=\e^{-s\pm \mathrm{i}\mu c}$, which do not lie on the unit circle as $s\not=0$. Hence $M$ is invertible and by inverting it we find suitable initial conditions $\beta_0$ and $\beta_1$ so that $\beta$ solves (\ref{betaeq}).

Secondly, we consider the case of a positive eigenvalue $\lambda^2$. Now a solution $\beta $ is given as 
		\[
		\beta(t)=\beta_0\cosh(\lambda t)+\tfrac{\beta_1}{\lambda}\sinh (\lambda t).
		\]
		The corresponding linear system for $\beta_0$ and $\beta_1$ that replaces (\ref{linsystrig}) now is given by the matrix
\[
M=
\begin{pmatrix}
\lambda ( \e^s \cosh(\lambda c) A -\1) &  \e^s \sinh(\lambda c)  A
\\[1mm]
\lambda \e^s  \sinh(\lambda c) A &  \e^s \cosh(\lambda c) A -\1
\end{pmatrix}
\]
with determinant 
\[
				\det(M) 
				= \lambda^n \det\Big(\e^{2s}A^2-2 \e^s\cosh(\lambda c)A+\1 \Big) = \prod_{i=1}^n\left( \e^{2s}z_i^2-2\e^s\cosh(\lambda c)z_i + 1\right).
\]
Again, the determinant of $M$ can only be zero if one of the $z_i$'s is a root of 
\[ \e^{2s}z^2-2\e^s\cosh(\lambda c)z+1.\]
Now the roots are  real numbers $z=\e^{-s\pm \lambda c}$ which do not lie on the unit circle {\em unless} $s=\pm \lambda c$. This however was excluded in the assumption and so  $M$ is invertible and yields a solution  $\beta$ to (\ref{betaeq}).

Finally, in the case of eigenvalue zero, the solution are affine,
$\beta(t)=\beta_1 t+\beta_0$,
so that (\ref{betaeq}) is equivalent to
\[(\e^sA  -\1)\beta_1=\hat\beta_1, \quad (\e^sA  -\1)\beta_0=\hat\beta_0-c\e^sA\beta_1.\]
Inverting $(\e^sA  -\1)$ gives the result also in this case.
\end{proof}

\begin{proposition}\label{prop_conjugating-good-homotheties}
			Let $\hat\phi\in H_S$ be a strict homothety such that $\epsilon_{\hat\phi}=1$. Unless 
			 $\left(\frac{s_{\hat\phi}}{c_{\hat\phi}}\right)^2$  is  an eigenvalue of $S$, the homothety
			 $\hat\phi$ is conjugate by an isometry from $\Hei_n\rtimes \Z_2$ to a strict homothety $\phi  
			 \in \E(1)\times  C_{\O(n)}(S)\times \R$ with $\epsilon_\phi=1$ and $c_\phi\ge 0$.
%
		\end{proposition}
		\begin{proof}
		Let $\hat\phi$ be a strict  homothety with $\epsilon_{\hat\phi}=1$. Then $\hat\phi=\hat\psi \circ h_{\hat{s}}$ with an isometry $\hat \psi=(\hat b,\hat \beta)\circ (\hat c,1,\hat A)$ with $(\hat b,\hat \beta)\in \Hei_n$ and $(\hat c,1,\hat A)\in \E(1)\times \cent_{\O(n)}(S)$, and $h_{\hat{s}}$ a pure homothety with $\hat s=s_{\hat\phi}\not=0$. 
		First, we are searching  for an isometry $\psi=(b,\beta)\in\Hei_n$ such that
		\begin{equation}\label{wanted}
		\psi\hat\phi\psi^{-1}= 	
				\psi\hat\psi h_{\hat s}\psi^{-1}	\in \E(1)\times C_{\O(n)}(S)\times \R.
				\end{equation}
Note that, since $\epsilon_{\hat \phi}=1$ we automatically have that
$\epsilon_{\psi\hat\phi\psi^{-1}}=1$.				
Denoting by \begin{equation}
\label{fis}
\vf_s:=\vf(0,1,\1,s)\end{equation}
with $\vf:\E(1)\times \cent_{\O(n)}(S)\times \R\to \Aut(\Hei_n)$ from Proposition~\ref{homoprop} and using 
				using~\ref{hprod}), we have
\[
		\psi\hat\phi\psi^{-1}= 	
				\psi\hat\psi h_{\hat s}\psi^{-1}	=\psi\hat\psi\vf_{\hat s}(\psi^{-1}) h_{\hat{s}}.\]
Hence condition (\ref{wanted} ) is equivalent to
		\[\psi\hat\psi\vf_{\hat{s}}(\psi^{-1}) \in \E(1)\times C_{\O(n)}(S).\]	
Now we compute using the formulas (\ref{product}), (\ref{inverse}) and (\ref{hprod}),
\begin{eqnarray*}
\psi\hat\psi\vf_{\hat{s}}(\psi^{-1})
&=&
 \left( \left( (b+ \hat b + 
 \omega(\beta, \hat \beta) , \beta+ \hat \beta\right) ,( 
  \hat c,1 , \hat A)\right)
  (- b\e^{2\hat s} ,-\e^{\hat{s}} \beta)
 \\
 &=&
 \left( b+\hat b- \e^{2\hat s}b +\omega(\beta,\hat\beta) -\e^{\hat s}\omega \left(\beta+\hat\beta,\hat A \,\beta\circ \sigma_{-\hat c} \right) , \beta +\hat\beta- \e^{\hat s} \hat A \beta\circ \sigma_{-\hat c} \right) \left(   \hat c, 
1, \hat A\right) ,
\end{eqnarray*}
where $\sigma_{-\hat c} (t)=t-\hat c$ denotes the shift by $-\hat c$. By Lemma~\ref{lem_beta-solving-conjugation-equation}, there is a solution $\beta\in V_S$ to
\[
 \e^{\hat s} \hat A\, \beta\circ \sigma_{-\hat c}-\beta =\hat\beta.
\]
Given this $\beta$, we can solve for $b$ such that the first entry in the above display is zero, i.e.~such that 
$\psi\hat\psi\vf_{\hat{s}}(\psi^{-1})\in\E(1)\times C_{\O(n)}(S)$. 
Finally, we can conjugate 
$\psi\hat\psi\vf_{\hat{s}}(\psi^{-1})$ to the required $\phi$  by $\epsilon\in \Z_2$ to achieve that $c_\phi=\epsilon \hat c\ge 0$.
		\end{proof}

This leads to the following result:

\begin{theorem}\label{thm_no-homoth-quotients-of-imaginary-CW}
Let $(M,g_S)$ be a Cahen-Wallach space 
and $\Gamma $ be a group of homotheties   acting properly discontinuously and cocompactly. Then $\Gamma$ is contained in the isometries, unless  $S$ has at least one positive eigenvalue and all elements in $\Gamma\setminus \Iso$ satisfy 
\begin{equation}
\label{bad} \left(\frac{s_\gamma}{c_\gamma}\right)^2 \le  \lambda_{\max}^2,\end{equation}
where $\lambda^2_{\max}$ is the largest positive eigenvalue.

In particular, $\Gamma$ is contained in the isometries if $(M,g_S)$ is of imaginary type.
		\end{theorem}
		\begin{proof}
	Assume for contradiction that $\Gamma$ acts properly discontinuous and contains a strict homothety $\gamma$, i.e.~with $s_\gamma\not=0$. By Corollary \ref{cor:fp}, $\gamma $ must have $\epsilon_\gamma=1$ and $c_\gamma\not=0$.
If all eigenvalues of $S$ are nonpositive (including zero), or if 
 $\left(\frac{s_\gamma}{c_\gamma}\right)^2> \lambda_{\max}^2$, 
 we can use Proposition~\ref{prop_conjugating-good-homotheties} to conjugate by an isometry $\Gamma$ to $\hat\Gamma$, which still acts properly discontinuous and cocompactly but now contains a  strict homothety with he same $s_\gamma$ and $c_\gamma$ but in  
$ \R\times \cent_{\O(n)}(S) \times \R$. Then we can apply Proposition~\ref{prop_good-homoth-quotient-implies-exponential-beta} to get a contradiction.
 Hence, unless all $\gamma\in \Gamma \setminus \Iso$ satisfy inequality (\ref{bad}),
we get that $\Gamma$ is contained in the isometries.

For $(M,g_S)$ of imaginary type, $S$ has no positive eigen value and hence $\Gamma$ is contained in the isometries.
\end{proof}		


To prove the remainder of Theorem~\ref{intro-thm_no-homoth-quotients-of-imaginary-CW}, recall from by Proposition~\ref{prop_cahen-kerbrat-weyl-curved-implies-no-conformals} that if   $(\R^{n+2},g_S)$ is conformally curved,  		  $\Gamma$ is a group of homotheties and by Theorem ~\ref{prop_cahen-kerbrat-weyl-curved-implies-no-conformals} 
a group of isometries.
			Then the metric endowed to the quotient is locally isometric to $(\R^{n+2},g_S)$ and hence locally symmetric, 
			so by Corollary~\ref{locsymcompconf}, the conformal group of the quotient is equal  to its isometry group.

	\subsection{Cocompact groups in the centraliser of an essential homothety}	\label{centralsec}
	In this section we are going to provide another non-existence result. We will show that given an essential conformal  transformation $\eta$ on a conformally curved Cahen-Wallach space, there is no  subgroup in the centraliser of $\eta$ that acts cocompactly and properly discontinuously. The motivation for this was explained in the introduction: if there was such a subgroup, then the essential conformal transformation would descend   to the compact quotient (see \cite{stuart-thesis} for details) and hence would provide a counterexample to the Lorentzian Lichnerowicz conjecture. The counterexamples to the conjecture in signatures beyond Lorentzian in \cite{frances12} are constructed in this way. Our result excludes this possibility.
	
Recall the definition of $H_S$ in (\ref{defHS}) 
and consider a pure homothety \[h_s=\diag (1, \e^s\1, \e^{2s})\in H_S.\] It is straightforward to compute its centraliser in $H_S$ as
\[\cent_{H_S}(h_s) =\E(1)\times \cent_{\O(n)}(S)  \times \R.\]
Furthermore, denote by $p$ the projection
\[
p:H_S\longrightarrow \E(1)\times \cent_{\O(n)}(S)  \times \R\ \simeq\  H_S/\Hei_n,\]
which is a group homomorphism with kernel $\Hei_n$. 
\begin{proposition}\label{prop_cent-essential-homothety-projection-to-cent-pure-injective}
Let $\eta\in H_S$ a strict homothety 	fixing the origin in $\R^{n+2}$ and with $\epsilon_\eta=1$ and let $C_{H_S}(\eta)$ be its centraliser in the homotheties. 
Then the group homomorphism
\[q:=p|_{\cent_{H_S}(\eta)}: \cent_{H_S}(\eta)\longmapsto \E(1)\times \cent_{\O(n)}(S)  \times \R,\]
is injective.

Moreover, let $\gamma_n=(b_n,\beta_n)\cdot (c_n,\epsilon_n,A_n,s_n)\in C_{H_S}(\eta)$ such that $c_n\to_{n\to\infty} 0$. Then $b_n\to0$, and $\beta_n(0)\to 0$.
			In particular, if $c_n\to 0$, then $\gamma_n(0)\to0$.

\end{proposition}
\begin{proof}
			First, observe that since $\eta(0)=0$ and $\epsilon_\eta=1$, we have that $c_\eta=0$ and 
		 $\eta=\psi h_{s}$ with \[\psi=(0,\beta_\eta)\cdot A_\eta\in V_S\rtimes_\vf C_{\O(n)}(S)\subset  \Hei_n\rtimes_\vf C_{\O(n)}(S)\] an isometry fixing the origin, i.e.~with $\beta_\eta(0)=0$.

			Let $\gamma=(b,\beta) \in\ker(q)= \Hei_n\cap C_{H_S}(\eta)$. 
			By (\ref{hprod}) we have that  $\gamma\in \cent_{H_S}(\eta)$ if and only if 
			\begin{equation}\label{isocen}
	\psi=		\gamma\cdot \psi \cdot\vf_s(\gamma^{-1}) =(b,\beta)\cdot \psi \cdot \left(-\e^{2s}b,-\e^s\beta\right),
			\end{equation}
where $\vf_s$ was defined in (\ref{fis}), and equivalently  
 by 
			(\ref{product}) that
			\[
			0=(1-\e^{2s})b  +\langle \beta (0),\dot\beta_\eta(0)\rangle
\quad\text{
			and }\quad			\e^sA_\eta\beta(t)  = \beta ( t)\quad\text{for all $t$}
			 .\]
Since $s\not=0$, the second equation implies that $\beta \equiv 0$, and with that the first gives $b =0$. Therefore, $\gamma=\id$ and $q$ is injective. 			
			
			To show the second part of the proposition, we first determine the inverse of $q$ on its image.
			For $(c, \epsilon  ,A ,r )=q (\gamma)$  in the image of $q$ we need to find $(b,\beta)\in \Hei_n$ such that 			$\gamma=\phi h_{r }\in \cent_{H_S}(h_s)$ with 
			\[\phi=(b,\beta)\cdot(c, \epsilon  ,A ) \in \Hei_n\rtimes_\vf (\E(1)\times \cent_{\O(n)}(S).\]
Since homotheties commute and using (\ref{hprod}) again,  $\eta=\gamma \eta \gamma^{-1}$ yields the equation
			\[
			\psi = 
			\phi h_{r } \psi h_s h_{r }^{-1} \phi^{-1} h_s^{-1}
			=
			\phi \vf_{r}(\psi) \vf_s(\phi^{-1}).\]
Using (\ref{product}) this can be seen to be equivalent to $A_\eta=A  A_\eta A ^\top$,
\[
(\1-\e^sA_\eta ) \beta  (t)  =\beta_\eta (t)-\e^{r }A \beta_\eta(\epsilon  (t -c ))
\]
and  
\[
b (1-\epsilon \e^{2s}) = - \e^{r } \omega\left((1-\epsilon\e^s A_\eta) \beta  ,A  \beta_\eta\circ E_{-\epsilon c, \epsilon}\right)
+\e^s \omega\left(\beta, A_\eta \beta  \right),
\]			
where we use (\ref{shiftomega}) and (\ref{Aomega}).
%
Evaluating these equations at $t=0$ and taking into account that $\beta_\eta(0)=0$			we get
\[
(\1-\e^sA_\eta) \beta  (0)  =-\e^{r }A \beta_\eta(-\epsilon  c ))
\]
and  
\begin{eqnarray*}
2b (1-\epsilon \e^{2s}) &=&
 - \e^{r }\left(  \langle (1-\epsilon \e^s A_\eta) \beta(0)  ,A \dot \beta_\eta( -\epsilon c)\rangle  - \langle (1-\epsilon\e^s A_\eta) \dot\beta(0)  ,A  \beta_\eta(-\epsilon c) \rangle\right) 
\\
&&{ }+\e^s \langle (\beta(0) , A_\eta \dot \beta(0)  \rangle.\end{eqnarray*}
If  $\gamma_n$ is a sequence as in the proposition, i.e.~with $c_n\to 0$, using that $s\not=0$, the first equation implies that $\beta_n		(0)\to 0$ and consequently the second implies that $b_n\to 0$. Hence, $\gamma_n(0)$ converges to $0$. 
		\end{proof}
		\begin{theorem}\label{thm_cent-fixes-zero-not-PD-cocompact}
			Let $\eta$ be a strict homothety in $H_S$ that fixes zero. Let $\Gamma$ be a subgroup of the centraliser of $\eta$ in $H_S$.
			Then $\Gamma$ does not act properly discontinuously and cocompactly.
		\end{theorem}
		\begin{proof}
			Assume for contradiction that $\Gamma$ is a subgroup of $\cent_{H_S}(\eta)$ acting properly discontinuously and cocompactly. Without loss of generality we can assume that $\epsilon_\eta=1$. If not, $\eta^2$ has this property and we still have that $\Gamma$ is contained in the centraliser of $\eta$. We will derive a contradiction to (PD1).

			We can apply Proposition~\ref{prop_cent-essential-homothety-projection-to-cent-pure-injective} to get an isomorphism between $\Gamma$ and a subgroup $\hat \Gamma=q(\Gamma)$ of $\E(1)\times C_{\O(n)}(S)  \times \R$.
			We briefly justify that $\hat\Gamma=q(\Gamma)$ is discrete when the homotheties are given the product topology: 
			
			Let $\Gamma$ be equipped with a topology such that the action of $\Gamma$ is continuous, i.e.~the map $(\gamma,x)\mapsto \gamma(x)$ is continuous. Note that the map $\Theta(\gamma,x) = (x,\gamma(x))$ is then also continuous. 
			Because of  (PD1), there is  an open set $U$ such that $\Theta^{-1}(U\times U) = \{\id\}\times U$. Then, with 
			with $\Theta$ being continuous, 
			 $\{\id\}$ is open in $\Gamma$ and hence $\Gamma$ is discrete. 
			Hence, since  the homotheties act continuously when given the product topology, $\Gamma$ is discrete with respect to this topology. Then we note that $p$ is a projection map and hence is an open map. Therefore $\hat\Gamma=q(\Gamma) = p(\Gamma)$ is discrete. We remark that this argument does not require us to claim that the compact-open topology on the homotheties coincides with the product topology.

			By the second part of Proposition~\ref{prop_cent-essential-homothety-projection-to-cent-pure-injective}, if we have $\gamma\in\Gamma$ such that $c_\gamma=0$, then $\gamma(0)=0$.
			Then by freeness of $\Gamma$, $\gamma=\id$. So for all non-identity elements $\gamma\in \Gamma$ we have $c_\gamma\neq 0$.
			Note that by Proposition~\ref{prop_torsion-implies-fixed-point}, $\Gamma$ has no torsion elements.
			Then we also conclude that $\epsilon_\gamma=1$ for all non-identity elements $\gamma\in \Gamma $, since otherwise $c_{\gamma^2}=0$.

			Now we observe that the projection 
			$$\rho:\hat \Gamma=q(\Gamma)\longrightarrow \R^2,\qquad \gamma=(c_\gamma,\epsilon_\gamma, A_\gamma,s_\gamma)\longmapsto (c_\gamma,s_\gamma)$$
			is an injective homomorphism. Indeed, its kernel is contained in the compact group $K=\Z_2\times C_{\O(n)}(S)$. Since $\hat\Gamma$ is discrete, any non-identity element in the kernel must be torsion, which would imply that $\Gamma$ does not act freely by Proposition~\ref{prop_torsion-implies-fixed-point}. Therefore $\rho$ is injective.

			Since $\rho$ is a projection map, it is also an open map, i.e.~$\rho(q(\Gamma))$ is a discrete subgroup of $\R^2$. Hence, using that $\Gamma$ cannot be cyclic and also act cocompactly by Lemma~\ref{lem_cyclic-isnt-cocompact}, $\Gamma$ is a discrete non-cyclic subgroup of $\R^2$. Therefore $\Gamma$ must contain a subgroup that is isomorphic to $\Z^2$.
Let $\gamma,\phi\in\Gamma$ be two generators of $\Z^2$, i.e.~such that $\<\gamma\>\cap\<\phi\>=\{\id\}$.
			By earlier in this proof, $c_\phi\neq0$, and so we take a sequence of rational numbers $p_n/q_n$ approaching $c_\gamma/c_\phi$. Then we consider the sequence $\gamma^{p_n}\phi^{-q_n}$. This is a sequence of elements for which the component $c_n$ approaches $0$. Hence by the second part in Proposition~\ref{prop_cent-essential-homothety-projection-to-cent-pure-injective}, $\gamma^{p_n}\phi^{-q_n}(0)\to 0$, contradicting {PD1} in the definition of proper discontinuity.
		\end{proof}

		\begin{theorem}\label{thm_no-centraliser-essential-quotient}
			A group of homotheties of a Cahen-Wallach space centralising an essential homothety cannot act properly discontinuously and cocompactly.
		\end{theorem}
		\begin{proof}Let $\eta $ be the essential homothety.
			By Theorem~\ref{thm_CW-essential-iff-fixed-point}, $\eta$  has a fixed point.
			Since Cahen-Wallach spaces are homogeneous, the isometry group acts transitively. We conjugate $\eta$ and  $\Gamma$ by the isometry that sends $0$ to the fixed point of $\eta$.
Note that			$\Gamma$ acts properly discontinuously and cocompactly if and only if its conjugate acts properly discontinuously and cocompactly.
			Hence, without loss of generality,  we assume  that $\eta$ fixes $0$.
			Since  $\Gamma$ centralises $\eta$, by \cref{thm_cent-fixes-zero-not-PD-cocompact}, $\Gamma$ does not act properly discontinuously and cocompactly.
		\end{proof}
Combining this result with 
Proposition~\ref{prop_cahen-kerbrat-weyl-curved-implies-no-conformals},
we obtain Theorem~\ref{thm_main-no-essential-quotient}.

		\subsection{Examples}
		\label{ex-sec}
In this last section we are going to illustrate some of the difficulties that arise when attempting to construct compact  quotients of Cahen-Wallach spaces by groups of conformal transformations. We start with some examples of isometric quotients.
	
\begin{example}[Compact isometric quotient of imaginary type]\label{ex_Isometric-quotient}
For Cahen-Wallach spaces of imaginary type, the function $\beta$ in an isometry is given by trigonometric functions. This makes it relatively straightforward to find groups of isometries that act
 properly discontinuously and cocompactly. For simplicity, let  $(\R^4,g_{-})$ be a conformally flat Cahen-Wallach space of imaginary type of dimension $4$.

Solutions to $\ddot\beta=S\beta$ are of the form $\veccy u\cos(t)+\veccy w\sin(t)$, where $\veccy u, \veccy w\in \R^2$.
Let  $\Gamma$ be generated by the following isometries, 
			\begin{align*}
				\gamma\begin{pmatrix}t\\\veccy x\\v\end{pmatrix} &:= \begin{pmatrix}t+\frac\pi2\\\veccy x\\v\end{pmatrix}, \qquad
					\eta\begin{pmatrix}t\\\veccy x\\v\end{pmatrix} := \begin{pmatrix}t\\\veccy x\\v+1\end{pmatrix},
					\end{align*}
			where $\x=(x^1,x^2)$ and 
			\begin{align*}
				\zeta\begin{pmatrix}t\\\veccy x\\v\end{pmatrix} &:= \begin{pmatrix}t\\\veccy x+ \beta(t) \\v - \<\beta(t) , \veccy x\>\end{pmatrix},\quad \text{ with }\quad\beta(t)=\begin{pmatrix} \cos(t)\\\sin(t)\end{pmatrix}
			\end{align*}
			Consider the diffeomorphism  $f:\R^4 \to (\R^4,g_{-})$ given by
			$$f\begin{pmatrix}u\\x\\y\\v\end{pmatrix} := \begin{pmatrix}u\\x\begin{pmatrix}\cos(u)\\\sin(u)\end{pmatrix}+y\begin{pmatrix}-\sin(u)\\\cos(u)\end{pmatrix}\\-v-xy\end{pmatrix},$$
			with inverse 
			$$f^{-1}\begin{pmatrix}t\\ x^1\\ x^2\\v\end{pmatrix} = \begin{pmatrix}t\\ x^1\begin{pmatrix}\cos(-t)\\\sin(-t)\end{pmatrix}+ x^2\begin{pmatrix}-\sin(-t)\\\cos(-t)\end{pmatrix}\\-v-\left( x^1\cos(-t)- x^2sin(-t)\right)\left(   x^1\sin(-t)+   x^2\cos(-t)\right)\end{pmatrix}.$$
			Then the action on $\R^4$ by $\Gamma$ is given by
			\begin{align*}
				f^{-1}\gamma f\begin{pmatrix}u\\x\\y\\v\end{pmatrix} &= \begin{pmatrix}u+\frac\pi2\\y\\-x\\v\end{pmatrix}, \qquad
					f^{-1}\eta f\begin{pmatrix}u\\x\\y\\v\end{pmatrix} = \begin{pmatrix}u\\x\\y\\v-1\end{pmatrix},&
				f^{-1}\zeta f\begin{pmatrix}u\\x\\y\\v\end{pmatrix} &= \begin{pmatrix}u\\x+1\\y\\v\end{pmatrix}.
			\end{align*}
The conjugation has removed  the dependence of the group action on $t$. $\Gamma$ acts properly discontinuously and cocompactly on $\R^4$: 
observe that
			$$\Lambda:= f\<\gamma^4,\ \zeta,\ \gamma^{-1}\zeta\gamma,\ \eta\>f^{-1} = \<2\pi e_1,e_2,e_3,-e_4\>$$
			is a subgroup of $f\Gamma f^{-1}$ of index $4$ and  a lattice isomorphic to $\Z^4$, so $\Lambda$ acts properly and cocompactly. Hence 
			$f\Gamma f^{-1}$ acts properly and cocompactly.
			Then we observe that $f\Gamma f^{-1}$ and $\Lambda$ both act freely. Hence $f\Gamma f^{-1}$ acts properly discontinuously and cocompactly on $\R^4$.
			Therefore, $\Gamma$ acts properly discontinuously and cocompactly on $(\R^4,g_{-})$ and  			
			$\R^4/\Gamma$ is fourfold covered by the  $4$-torus $\R^4/\Lambda$.
		\end{example}

\begin{example}[Compact isometric quotient of real type]\label{example} 
Here we give an example of a group $\Gamma$ 
acting properly discontinuously and cocompact by isometries on a conformally flat Cahen-Wallach space of real type and  of dimension $4$. We will also show why an attempt to generalise this to a group of homotheties fails. We follow the construction in  \cite[Chapter~5]{KathOlbrich15}.
For $r\in \N_{\ge 3}$, consider the polynomial
\[f(x)=x^2-rx+1,\] 
and let $\rho\not=\tfrac{1}{\rho}$ be its roots. 
Let $(\R^4,g_\rho)$ be the conformally flat Cahen-Wallach space of real type defined by $g_\rho:=g_S$ with $S=(\ln|\rho|)^2\1$, i.e.~
\[ g_\rho =2\d t \d v + (\ln|\rho|)^2(x^2+y^2)\d t^2 + \d x^2 +\d y^2.\]

According to  \cite[Proposition~8.8]{KathOlbrich15}, $(\R^4,g_\rho)$ admits a subgroup of the transvections,  $\Gamma \subset \Hei_2\rtimes_\vf\R$, such that $\R^4/\Gamma$ is a compact manifold. 
Such $\Gamma$ can be given as follows: 

Let 
\[
\beta (t)=\begin{pmatrix} \rho^t\\[1mm] \rho^{-t}\end{pmatrix} \quad \text{ and }\quad \hat\beta(t)=\beta (t+1)
\]
 be solutions to $\ddot\beta=S\beta$ and denote the corresponding isometries also by $\eta$ and $\hat\eta$. Let $\alpha_b$ be the translation in the $v$-component by $ b$ and  define $\gamma_c$ as the translation by $c$ in the $t$-component, 
\[
 \begin{pmatrix}t\\x\\y \\v\end{pmatrix} \stackrel{\alpha_b}{\mapsto}
 \begin{pmatrix}t
  \\
   x 
\\
 y
 \\
v+ b
 \end{pmatrix},\quad 
  \begin{pmatrix}t\\x\\y \\v\end{pmatrix} \stackrel{\gamma_c}{\mapsto}
 \begin{pmatrix}t+c
  \\
   x 
\\
 y
 \\
v
 \end{pmatrix}
 \]
 and let $\Gamma$ be the group of isometries generated by $\alpha:=\alpha_1$, $\eta$, $\hat \eta$ and $\gamma:=\gamma_1$. An arbitrary group element in $\Gamma$ is given as 
 \[
\begin{pmatrix}t\\x\\y \\v\end{pmatrix} \stackrel{\phi}{\mapsto}
 \begin{pmatrix}t+ k \\[1mm] x +\rho^{t}(n +m\rho)
\\[1mm]
y+\rho^{-t}(n +\frac{m}{\rho})
 \\[1mm]
  v+l -  \ln|\rho | \left( (n+m\rho)x \rho^{t} +\frac{1}{2}(n+m\rho)^2\rho^{2t}
 + (n+\frac{m}{\rho})y \rho^{-t} +\frac{1}{2}(n+\frac{m}{\rho})^2 \rho^{-2t}\right)
 \end{pmatrix}
\]
with $k,l,m,n\in \Z$.
 In order to show that $\Gamma$ acts cocompactly and properly discontinuously, we use  \cite[Proposition~4.8]{KathOlbrich15}).
First we note that 
\[
\beta(0)= \begin{pmatrix}1\\1\end{pmatrix},\quad \hat\beta(0)= \begin{pmatrix}\rho\\\tfrac{1}{\rho}\end{pmatrix}
\]
are linearly independent. This is condition (a) in \cite[Proposition~4.8]{KathOlbrich15}). Note that  $L=\Span_\R(\beta,\hat\beta)$ is  a Lagrangian subspace for $\omega$, i.e.~$\omega(\beta,\hat\beta)=0$. 
For condition (b)  \cite[Proposition~4.8]{KathOlbrich15}) we need to find a lattice $\Lambda$ in $\R\times L\subset \Hei_2$ that is invariant under the shift $\tau: t\mapsto t+1$. Note that 
\[\hat\beta(t+1)-r\hat\beta(t) +\beta(t)
=
\begin{pmatrix}
\rho^{t}  f(\rho)\\[2mm]
\rho^{-t} f(\rho^{-1})
\end{pmatrix} =0,
\]
so that 
\[\tau(\eta)=\hat\eta,\quad \tau(\hat\eta)= -\eta+r\hat\eta,\quad \tau^{-1}(\hat\eta)=\eta,\quad \tau^{-1}(\eta)= r\eta-\hat\eta.\]
Hence,
the lattice $\Lambda_0=\Span_\Z\{\eta,\hat\eta\}$ is stable under the action of the group $\<\tau\>$ and so is the lattice $\Lambda=\Span_\Z\{\alpha,\eta,\hat\eta\}$ in $\R\times L\subset \Hei_2$.
Hence, by  \cite[Proposition~4.8]{KathOlbrich15}), $\Gamma$ is a group of isometries that acts properly discontinously and cocompactly on $\R^4$.

In order to generalise this to a group $\hat \Gamma$ of homotheties acting cocompactly and properly discontinuously one could try to replace $\alpha$ by the translation $\alpha_\rho$ by $\rho$ in the $v$-component and $\gamma$ by the homothety $\hat \gamma$, 
 \[
\begin{pmatrix}t\\x\\y \\v\end{pmatrix} \stackrel{\gamma_c}{\mapsto}
 \begin{pmatrix}t+ 1
  \\
  \rho x 
\\
\rho y
 \\
  \rho^{2} v
 \end{pmatrix}
\]
and  $\hat\Gamma$ be the subgroup of $H= (\Hei\rtimes_\vf(\R\times\R)$ that is generated by $\alpha_\rho$, $\eta$, $\hat\eta $ $\hat\gamma$,
\[\Gamma=\<\alpha, \eta,\hat\eta,\hat\gamma\>.\]
Then one may try use
ideas in \cite[Chapters 3 \& 4]{KathOlbrich15}
 to show that $\Gamma$ acts properly discontinuously and cocompactly on $\R^4$.  
For this we need a group $G$ that is a semidirect product of a nilpotent group with $\R$. 
Note that also $\Hei_2$ is invariant under conjugation with 
$\R \hat\gamma$
in $H= (\Hei\rtimes_\vf(\R\times\R)$  and we set 
\[G=
\Hei_2\rtimes_{\vf}\R\hat\gamma.\]
Note that even though $\Hei_2$ is normal in $G$, $\Lambda_3$ is not normal in $\Gamma$, in fact
\[
\hat{\gamma}^{k}( m\alpha+n\eta+\hat n\hat\eta) \hat{\gamma}^l= 
\hat{\gamma}^k\left( m\alpha +n\sigma_l\eta+\hat n\sigma_l\hat\eta+ \hat{\gamma}_{l}\right)
=
m\alpha_{\rho^{1+2k}} +n\rho^{k}\sigma_l\eta+\hat n\rho^{k}\sigma_l\hat\eta+ \hat{\gamma}_{k+l}.
\]
This shows that the group $\Gamma$ is not discrete. Indeed, the sequence $\hat{\gamma}^{-k}\alpha\hat{\gamma}^k=\alpha_{\rho^{1-2k}} $ for $k\in \N$  converges to the identity. 
 \end{example}

 	\begin{example}\label{ex_homoth-no-real-type-quotient}
			The previous example demonstrates the issues that arise when attempting to modify a properly discontinuous and cocompact group of isometries of a Wey-flat Cahen-Wallach space of real type to a group of homotheties by maintaining the translations in the $v$-direction.  In this example we will try a different approach that avoids these translations.

	For simplicity, let $(\R^3,g_1)$ be a three dimensional Cahen-Wallach space of real type. Solutions to $\ddot \beta = S\beta$ are of the the form $a \e^{t}+ b \e^{-t}$, where $a,b\in\R$.

			As before, we consider the homothety $\gamma$ 	and the isometry  $\eta$
			$$\gamma:\begin{pmatrix}t\\  x\\v\end{pmatrix}\mapsto \begin{pmatrix}t+1\\e  x\\e^2v\end{pmatrix},\quad \eta: \begin{pmatrix}t\\  x\\v\end{pmatrix}\mapsto \begin{pmatrix}x\\  x + k\e^{t}\\v - \<k\e^{t},  x + \frac12 k\e^{t}\>\end{pmatrix}.$$
		We define a diffeomorphism $f:\R^3\to \R^3$ and its inverse by
			$$f:\begin{pmatrix}t\\y\\z\end{pmatrix}\mapsto \begin{pmatrix}t\\e^ty\\e^{2t}(z-y^2/2)\end{pmatrix},\qquad
f^{-1}:\begin{pmatrix}t\\ x \\v\end{pmatrix}\mapsto \begin{pmatrix}t\\e^{-t} x\\e^{-2t}(v+ x^2/2)\end{pmatrix}.$$
			The conjugates are
			$$f^{-1}\gamma f:\begin{pmatrix}x\\y\\z\end{pmatrix}\mapsto \begin{pmatrix}x+1\\y\\z\end{pmatrix},\qquad f^{-1}\eta f:\begin{pmatrix}x\\y\\z\end{pmatrix}\mapsto \begin{pmatrix}x\\y+k\\z\end{pmatrix}.$$
			At this stage it looks promising, but we still have a remaining direction to compactify. The issue that occurs in general at this stage is that when we have  a strict homothety $\gamma$ of the simplest form possible without admitting fixed points as in Proposition~\ref{prop_homothetic-fixed-points}, introducing an element $\alpha$ that translates in the $v$-direction will not help us, for the same reason as in the previous example: $\gamma^{-i}\alpha\gamma^i(0)$ will approach $0$. What this means is that it seems we will require $\beta$-terms to compactify in $n+1$ directions. However this turns out to be difficult:			if
			$$\zeta: \begin{pmatrix}t\\ x\\v\end{pmatrix}\mapsto \begin{pmatrix}t\\ x + k\e^{-t}\\v - k\e^{-t}\left( x + \frac12 k\e^{-t}\right)\end{pmatrix},$$
			then
			$$f^{-1}\zeta f: \begin{pmatrix}x\\y\\z\end{pmatrix}\mapsto \begin{pmatrix}x\\y+l\e^{-2x}\\z\end{pmatrix}.$$
			This  demonstrates two issues: first that the conjugated element fails to act at all on the $z$ direction --- which is the direction still remaining to be compactified. And secondly, that when we have a homothety in the form of $\gamma$, we see immediately that only $n$ of the $\beta$ dimensions are able to grow fast enough to avoid $\gamma^{-i}\zeta\gamma^i(0)\to 0$. This is because it is not sufficient that $\beta$ be exponential, it must grow exponentially in the same direction as $\gamma$.
			So this example too cannot lead to a properly discontinuous and cocompact action.
%
		\end{example}

		\begin{example}[Compact homothetic quotient of an open subset]\label{ex_removed-homoth-fixed-points}
			In this example, we produce a compact quotient of an open submanifold of a Cahen-Wallach space  by homotheties. 

			Consider the metric $g_S$ on 
 $U:= \R^{n+2}\setminus\{(t,0,0)\ |\ t\in\R\}$. We have removed all fixed points of a pure homothety, allowing us to use a pure homothety to compactify:
			set 
			$$\gamma\begin{pmatrix}t\\\veccy x\\ v\end{pmatrix}:= \begin{pmatrix}t+1\\ \veccy x \\ v\end{pmatrix}, \qquad \eta\begin{pmatrix}t\\\veccy x\\ v\end{pmatrix}:= \begin{pmatrix}t\\ 2\veccy x \\ 4 v\end{pmatrix},$$
			and $\Gamma:= \<\gamma,\eta\>$.
			We now show that $\Gamma$ acts properly discontinuously and cocompactly on $U$.

			A fundamental region for this action is a product of the unit interval  and an annulus in the last $n+1$ dimensions.
			Define
			$$R:= (0,1)\times \left((-2,2)^n\times(-4,4)\right)\setminus[-1,1]^{n+1}.$$
It is not hard to see that $\phi(R)$,  $\phi\in \Gamma$ does not meet $R$, 
%
			so $R$ is a fundamental region.
			We take a neighbourhood $V$ of $\overline R$:
			$$V:=(-1,2)\times \left((-4,4)^n\times(-16,16)\right)\setminus\left([-\tfrac12,\tfrac12]^n\times[-\tfrac14,\tfrac14]\right).$$
			Note that $\phi(V)$ meets $V$ only for
			$$\{\gamma^i\eta^j\ |\ i,j\in\{-2,-1,0,1,2\}\}.$$
			Hence $R$ is finitely self adjacent.
			In particular, by results in \cite{FSApaper,stuart-thesis}, $R$ is locally finite, so $U/\Gamma$ is homeomorphic to $\overline R/\Gamma$. Then, since $\overline R/\Gamma$ is a manifold  we get that $\Gamma$ acts properly discontinuously and since $\overline R$ is compact, $\Gamma$ also acts cocompactly.
			Hence the action on the open submanifold $U$ is properly discontinuous and cocompact. However, the homotheties centralised by $\Gamma$ are not essential. We have
			$$\cent_{H_S}(\Gamma) =
			\R\times C_{\O(n)}(S)\times \R
$$
and 			define
			$$f(t,\veccy x, v) = (||\veccy x||^4 + (v)^2)^{-1/2}.$$
			Then for $\phi\in \cent_{H_S}(\Gamma)$,
\[
				\phi^*(fg)|_x 
					= (||\e^sA\veccy x||^4 + (\e^{2s}v)^2)^{-1/2} \e^{2s}g|_x
					= (||\veccy x||^2 + (v)^2)^{-1/2} g|_x
					= (fg)|_x,
\]
			so $\phi$ is inessential on $U$. Note that this same choice of $f$ works for all such $\phi$, and thus $\cent_{H_S}(\Gamma)$ is inessential.
			In this example we can go further and conclude that the normaliser of $\Gamma$ is inessential as well, since  $\phi\gamma^r\eta^t = \gamma^{r'}\eta^{t'}\phi$ implies already that $t=t'$, and that $r' = ar$. Hence, we see that the normaliser is simply $\E(1)\times C_{\O(n)}(S)\times \R$, and the same $f$ as before makes the normaliser inessential.

			We stress that this is not a proof that $U/\Gamma$ has an inessential conformal structure: it is possible to have an essential transformation on the quotient which does not lift or whose lift is not essential, or  a transformation may be preserved without normalising $\Gamma$. For details see \cite[Section~5.4]{stuart-thesis}.
		\end{example}

\begin{remark} In this final remark we address the fact that 
our results are about compact quotients of a {\em complete} Cahen-Wallach space $(\R^{n+2},g_S)$, whereas
 the construction in \cite{frances12} starts with an {\em incomplete} locally symmetric space that has the origin removed.
For this, note that every non-isometric  conformal transformation of a conformally curved Cahen-Wallach space is a homothety and hence either has no fixed points or it has a line of finite-orbit points parameterised by $t$.
Now assume that we have  two strict  homotheties $\gamma$ and  $\phi$,  each with finite-orbit points, such that $\phi$ descends to the quotient by a group $\Gamma$ containing $\gamma$. Then $\gamma$ and  $\phi$  must have the same line of finite-orbit points, since $\phi$ must map finite-orbit points of $\gamma$ to finite-orbit points of $\gamma$. Consequently, 
 if $\Gamma $ is acting  on an open subset of of a Cahen-Wallach space that has the finite-orbit points removed, so that it acts properly discontinuous, then also  $\phi$ has had its finite-orbit points removed and therefore can no longer be expected to be essential.
\end{remark}

\bibliographystyle{abbrv}
\bibliography{GEOBIB}

\providecommand{\MR}[1]{}\def\cprime{$'$} \def\cprime{$'$} \def\cprime{$'$}
\begin{thebibliography}{10}

\bibitem{Alekseevski85}
D.~Alekseevski.
\newblock Self-similar {L}orentzian manifolds.
\newblock {\em Ann. Global Anal. Geom.}, 3(1):59--84, 1985.

\bibitem{Alekseevskii72}
D.~V. Alekseevski\u{\i}.
\newblock Groups of conformal transformations of {R}iemannian spaces.
\newblock {\em Mat. Sb. (N.S.)}, 89(131):280--296, 356, 1972.

\bibitem{besse87}
A.~L. Besse.
\newblock {\em {E}instein Manifolds}.
\newblock Springer Verlag, Berlin-Heidelberg-New York, 1987.

\bibitem{Cahen98}
M.~Cahen.
\newblock Lorentzian symmetric spaces.
\newblock {\em Acad. Roy. Belg. Bull. Cl. Sci. (6)}, 9(7-12):325--330, 1998.

\bibitem{CahenKerbrat77}
M.~Cahen and Y.~Kerbrat.
\newblock Transformations conformes des espaces sym\'etriques
  pseudo-riemanniens.
\newblock {\em C.~R. ~Acad.~Sci.~Paris S\'er. A-B}, 285(5):A383--A385, 1977.

\bibitem{CahenKerbrat78}
M.~Cahen and Y.~Kerbrat.
\newblock Champs de vecteurs conformes et transformations conformes des espaces
  lorentziens sym\'etriques.
\newblock {\em J. Math. Pures Appl. (9)}, 57(2):99--132, 1978.

\bibitem{CahenKerbrat82}
M.~Cahen and Y.~Kerbrat.
\newblock Transformations conformes des espaces sym\'etriques
  pseudo-riemanniens.
\newblock {\em Ann. Mat. Pura Appl. (4)}, 132:275--289 (1983), 1982.

\bibitem{CahenKerbratPraet75}
M.~Cahen, Y.~Kerbrat, and G.~Praet.
\newblock Conformational completion of {L}orentz symmetric spaces.
\newblock {\em Lett. Math. Phys.}, 1(5):417--422, 1975/77.

\bibitem{CahenKerbratPraet76}
M.~Cahen, Y.~Kerbrat, and G.~Praet.
\newblock Compl\'etions \'equivariantes conformes d'espaces lorentziens
  sym\'etriques.
\newblock {\em Acad. Roy. Belg. Bull. Cl. Sci. (5)}, 62(10):767--783, 1976.

\bibitem{cahen-wallach70}
M.~Cahen and N.~Wallach.
\newblock {L}orentzian symmetric spaces.
\newblock {\em Bull. Amer. Math. Soc.}, 79:585--591, 1970.

\bibitem{CalabiMarkus62}
E.~Calabi and L.~Markus.
\newblock Relativistic space forms.
\newblock {\em Ann. of Math. (2)}, 75:63--76, 1962.

\bibitem{carriere89}
Y.~Carri{\`e}re.
\newblock Autour de la conjecture de {L}. {M}arkus sur les vari\'et\'es
  affines.
\newblock {\em Invent. Math.}, 95(3):615--628, 1989.

\bibitem{Ferrand96}
J.~Ferrand.
\newblock The action of conformal transformations on a {R}iemannian manifold.
\newblock {\em Math. Ann.}, 304(2):277--291, 1996.

\bibitem{Frances05}
C.~Frances.
\newblock Sur les vari\'et\'es lorentziennes dont le groupe conforme est
  essentiel.
\newblock {\em Math. Ann.}, 332(1):103--119, 2005.

\bibitem{frances08}
C.~Frances.
\newblock Essential conformal structures in {R}iemannian and {L}orentzian
  geometry.
\newblock In {\em Recent developments in pseudo-{R}iemannian geometry}, ESI
  Lect. Math. Phys., pages 231--260. Eur. Math. Soc., Z\"urich, 2008.

\bibitem{frances12}
C.~Frances.
\newblock About pseudo-{R}iemannian {L}ichnerowicz conjecture.
\newblock {\em Transform. Groups}, 20(4):1015--1022, 2015.

\bibitem{FrancesMelnick10}
C.~Frances and K.~Melnick.
\newblock Conformal actions of nilpotent groups on pseudo-{R}iemannian
  manifolds.
\newblock {\em Duke Math. J.}, 153(3):511--550, 2010.

\bibitem{FrancesMelnick21}
C.~Frances and K.~Melnick.
\newblock The {L}orentzian {L}ichnerowicz {C}onjecture for real-analytic,
  three-dimensional manifolds, 2021.
\newblock arXiv:2108.07215.

\bibitem{KathOlbrich15}
I.~Kath and M.~Olbrich.
\newblock Compact quotients of {C}ahen-{W}allach spaces.
\newblock {\em Mem. Amer. Math. Soc.}, 262(1264):v+84, 2019.

\bibitem{klingler96}
B.~Klingler.
\newblock Compl\'etude des vari\'et\'es lorentziennes \`a courbure constante.
\newblock {\em Math. Ann.}, 306(2):353--370, 1996.

\bibitem{KuhnelRademacher95}
W.~K{\"u}hnel and H.-B. Rademacher.
\newblock Essential conformal fields in pseudo-{R}iemannian geometry.
\newblock {\em J. Math. Pures Appl. (9)}, 74(5):453--481, 1995.

\bibitem{KuhnelRademacher97}
W.~K{\"u}hnel and H.-B. Rademacher.
\newblock Essential conformal fields in pseudo-{R}iemannian geometry. {II}.
\newblock {\em J. Math. Sci. Univ. Tokyo}, 4(3):649--662, 1997.

\bibitem{KuhnelRademacher08}
W.~K{\"u}hnel and H.-B. Rademacher.
\newblock Conformal transformations of pseudo-{R}iemannian manifolds.
\newblock In {\em Recent developments in pseudo-{R}iemannian geometry}, ESI
  Lect. Math. Phys., pages 261--298. Eur. Math. Soc., Z\"urich, 2008.

\bibitem{MR31310}
N.~H. Kuiper.
\newblock On conformally-flat spaces in the large.
\newblock {\em Ann. of Math. (2)}, 50:916--924, 1949.

\bibitem{Kulkarni81}
R.~S. Kulkarni.
\newblock Proper actions and pseudo-{R}iemannian space forms.
\newblock {\em Adv. in Math.}, 40(1):10--51, 1981.

\bibitem{leistner-schliebner13}
T.~Leistner and D.~Schliebner.
\newblock Completeness of compact {L}orentzian manifolds with abelian holonomy.
\newblock {\em Math. Ann.}, 364(3-4):1469--1503, 2016.

\bibitem{FSApaper}
T.~Leistner and S.~Teisseire.
\newblock Fundamental regions for non-isometric group actions, 2021.
\newblock arXiv:2112.11082.

\bibitem{Lelong-Ferrand71}
J.~Lelong-Ferrand.
\newblock Transformations conformes et quasi-conformes des vari\'et\'es
  riemanniennes compactes (d\'emonstration de la conjecture de {A}.
  {L}ichnerowicz).
\newblock {\em Acad. Roy. Belg. Cl. Sci. M\'em. Collect. 8(2)}, 39(5):44, 1971.

\bibitem{Lichnerowicz64}
A.~Lichnerowicz.
\newblock Sur les transformations conformes d'une vari\'{e}t\'{e} riemannienne
  compacte.
\newblock {\em C. R. Acad. Sci. Paris}, 259:697--700, 1964.

\bibitem{Melnick21}
K.~Melnick.
\newblock Rigidity of transformation groups in differential geometry.
\newblock {\em Notices Amer. Math. Soc.}, 68(5):721--732, 2021.

\bibitem{MelnickPecastaing21}
K.~Melnick and V.~Pecastaing.
\newblock The conformal group of a compact simply connected {L}orentzian
  manifold.
\newblock {\em J. Amer. Math. Soc.}, 35(1):81--122, 2021.

\bibitem{Obata71}
M.~Obata.
\newblock The conjectures on conformal transformations of {R}iemannian
  manifolds.
\newblock {\em J. Differential Geometry}, 6:247--258, 1971.

\bibitem{oneill83}
B.~O'Neill.
\newblock {\em Semi-Riemannian Geometry}.
\newblock Academic Press, 1983.

\bibitem{Pecastaing17}
V.~Pecastaing.
\newblock Essential conformal actions of {$\mathrm{PSL}(2,\bold{R})$} on
  real-analytic compact {L}orentz manifolds.
\newblock {\em Geom. Dedicata}, 188:171--194, 2017.

\bibitem{Pecastaing18}
V.~Pecastaing.
\newblock Lorentzian manifolds with a conformal action of {SL}(2,{R}).
\newblock {\em Comment. Math. Helv.}, 93(2):401--439, 2018.

\bibitem{Podoksenov89}
M.~N. Podoks{\"e}nov.
\newblock A {L}orentzian manifold with a one-parameter group of homotheties
  that has a closed isotropic orbit.
\newblock {\em Sibirsk. Mat. Zh.}, 30(5):135--137, 217, 1989.

\bibitem{Podoksenov92}
M.~N. Podoks{\"e}nov.
\newblock Conformally homogeneous {L}orentzian manifolds. {II}.
\newblock {\em Sibirsk. Mat. Zh.}, 33(6):154--161, 232, 1992.

\bibitem{Ratcliffe06}
J.~G. Ratcliffe.
\newblock {\em Foundations of hyperbolic manifolds}, volume 149 of {\em
  Graduate Texts in Mathematics}.
\newblock Springer, New York, second edition, 2006.

\bibitem{Silvester00}
J.~R. Silvester.
\newblock Determinants of block matrices.
\newblock {\em The Mathematical Gazette}, 84(501):460--467, 2000.

\bibitem{stuart-thesis}
S.~Teisseire.
\newblock Conformal group actions on {C}ahen-{W}allach spaces.
\newblock Master's thesis, University of Adelaide, School of Mathematical
  Sciences, 2021.
\newblock http://hdl.handle.net/2440/131752.

\bibitem{teschl12}
G.~Teschl.
\newblock {\em Ordinary Differential Equations and Dynamical Systems}, volume
  140 of {\em Graduate Studies in Mathematics}.
\newblock American Mathematical Society, Providence, RI, 2012.

\end{thebibliography}
\end{document}